 \documentclass[11pt,a4paper,oneside,reqno]{amsart}
\usepackage{amsmath}
\usepackage{amssymb}
\usepackage{amstext}
\usepackage{amsfonts}
\usepackage{amsthm}
\usepackage{ifthen}
\usepackage{xspace}
\usepackage{comment}
\usepackage{graphicx}
\usepackage{mathrsfs}
\usepackage{enumerate}
\usepackage{breakcites}
\usepackage{xcolor}

\usepackage{hyperref}
\numberwithin{equation}{section}  

\newtheorem{punkt}{}[section]

\theoremstyle{plain}
\newtheorem{corollary}[punkt]{Corollary}
\newtheorem{lemma}[punkt]{Lemma}

\newtheorem{theorem}[punkt]{Theorem}

\theoremstyle{definition}
\newtheorem{remark}[punkt]{Remark}

\newtheorem{examples}[punkt]{Examples}

\newtheorem{definition}[punkt]{Definition}

\theoremstyle{plain}
\newtheorem*{corollary*}{Corollary}
\newtheorem*{lemma*}{Lemma}
\newtheorem*{proposition*}{Proposition}
\newtheorem*{theorem*}{Theorem}

\theoremstyle{definition}
\newtheorem*{remark*}{Remark}
\newtheorem*{remarks*}{Remarks}
\newtheorem*{example*}{Example}
\newtheorem*{examples*}{Examples}
\newtheorem*{problem*}{Problem}
\newtheorem*{problems*}{Problems}
\newtheorem*{question*}{Question}
\newtheorem*{questions*}{Questions}
\newtheorem*{definition*}{Definition}
\newtheorem*{conjecture*}{Conjecture}
\newtheorem*{assumption*}{Assumption}
\newtheorem*{assumptions*}{Assumptions}
\newtheorem*{construction*}{Construction}

\def\mycmplx{\mathbb{C}}

\def\mynat{\mathbb{N}}

\def\myreal{\mathbb{R}}

\def\mys{\mathcal{S}}

\def\tr{\qopname\relax{no}{tr}}

\def\re{\qopname\relax{no}{Re}\,}
\def\im{\qopname\relax{no}{Im}\,}

\def\diag{\qopname\relax{no}{diag}}

\def\eg{e.g.\@\xspace}
\def\ie{i.e.\@\xspace}

\def\GL{\qopname\relax{no}{GL}}

\def\Pos{\qopname\relax{no}{Pos}}

\def\fourier{\mathcal{F}}

\def\mellin{\mathcal{M}}
\def\harish{\mathcal{H}}
\def\FSV{f_{\operatorname{SV}}}		% should we better call it SSV ?
\def\FEV{f_{\operatorname{EV}}}

\newcommand{\eins}{\leavevmode\hbox{\small1\kern-3.8pt\normalsize1}}

\begin{document}

\title[Singular Value and Eigenvalue Statistics]{Exact Relation between Singular Value and Eigenvalue Statistics}
\author{Mario Kieburg$^{1,2,*}$ and Holger K\"osters$^{3,\dagger}$}
\address{$^1$Faculty of Physics \\
University of Duisburg-Essen \\
Lotharstr. 1, D-47048 Duisburg,
Germany}
\address{$^2$Department of Physics\\
Bielefeld University\\
Postfach 100131,
D-33501 Bielefeld, Germany}
\address{$^3$Department of Mathematics\\
Bielefeld University\\
Postfach 100131,
D-33501 Bielefeld, Germany}
\email{$^*$ mkieburg@physik.uni-bielefeld.de}
\email{$^\dagger$ hkoesters@math.uni-bielefeld.de}
\date{\today}

\maketitle

\begin{abstract}
We use classical results from harmonic analysis on matrix spaces to investigate 
the relation between the joint densities of the singular values and 
the eigenvalues for complex random matrices which are bi-unitarily invariant 
(also known as isotropic or unitary rotation invariant). 
We prove that one of these joint densities determines the other one. Moreover
we construct an explicit formula relating both joint densities at finite matrix dimension. 
This relation covers probability densities as well as signed densities.
With the help of this relation we derive general analytical relations 
among the corresponding kernels and biorthogonal functions for a specific class of polynomial ensembles. 
Furthermore we show how to generalize the relation between the eigenvalue and singular value statistics 
to certain situations when the ensemble is deformed by a term which breaks the bi-unitary invariance.
\end{abstract}


\begin{small}
\tableofcontents
\end{small}

% ****************************************************************************************************

\section{Introduction}
\label{sec:Introduction}

It is a long standing problem to directly relate the eigenvalues and the singular values of a general matrix irrespective of whether the matrix is fixed or random. Already the establishment of the Haagerup-Larsen theorem~\cite{HL-theorem,HS:2007} and of the single ring theorem~\cite{SR-theorem-1,SR-theorem-2} for the macroscopic level density of a certain class of random matrices are highly non-trivial.  The reason for this is that, apart from the~equality
\begin{equation}\label{det-equ}
|\det g|^2=|\det z|^2=\det a,
\end{equation}
only inequalities are known for an arbitrary fixed or random matrix  $g\in\mathbb{C}^{n\times n}$, 
e.g. Weyl's inequalities~\cite{Weyl}
\begin{equation}\label{inequalities-Weyl}
\prod_{j=1}^k |z_j|^2\leq \prod_{j=1}^k a_j,\ {\rm for\ all\ }k=1,\ldots,n.
\end{equation}
The eigenvalues $z=\diag(z_1,\ldots,z_n)$ of the matrix $g$ and its squared singular values $a=\diag(a_1,\ldots,a_n)$ (eigenvalues of $gg^*$) are ordered such that $|z_1|\geq\ldots\geq|z_n|$ and $a_1\geq \ldots\geq a_n$. 
Here we identify vectors with diagonal matrices, a convention which will be used throughout the present work
as it is convenient for many formulas involving matrices such as Eq.~\eqref{det-equ}.
Horn~\cite{Horn} has proven that one can invert the statement above, 
\ie if $z=\diag(z_1,\ldots,z_n)$ and $a=\diag(a_1,\ldots,a_n)$ satisfy Eqs.~\eqref{det-equ} and \eqref{inequalities-Weyl}, 
then there exists a matrix $g$ which has the variables $z_j$ as its complex eigenvalues and the variables $a_j$ as its squared singular values.
Furthermore, Weyl's inequalities \eqref{inequalities-Weyl} imply a number of related inequalities
\cite{Weyl} such as
\begin{equation}\label{inequalities-Horn}
\sum_{j=1}^k |z_j|^2\leq \sum_{j=1}^k a_j,\ {\rm for\ all\ }k=1,\ldots,n.
\end{equation}
For $k=n$, Eq.~\eqref{inequalities-Horn} was already proven by Schur \cite{Schur}.
In the special case that the matrix $g$ is normal (\ie $gg^*=g^*g$ with $g^*$ the Hermitian adjoint of $g$),
the inequalities~\eqref{inequalities-Weyl} and \eqref{inequalities-Horn} become equalities 
meaning $a_j=|z_j|^2$ for all $j=1,\ldots,n$. However for general matrices such equalities 
do not necessarily hold.
 
This situation drastically changes when the matrix $g$ is drawn from a random matrix ensemble. The spectral statistics of the induced Ginibre or Laguerre ensemble \cite{Ginibre,Akemann-book,IK:2013} (Gaussian) and the induced Jacobi~\cite{ZS-trunc,Akemann-book,Forrester-book,IK:2013} (truncated unitary) ensemble hint to a relation between the two statistics. Even for the generalization to products of these kinds of random matrices and their inverses simple relations are near at hand, see \eg Refs.~\cite{ARRS:2013,AB:2012,AIK:2013,AKW:2013,ABKN,Forrester:2014,IK:2013,KS:2014,KZ:2014,KKS:2015} and a recent review~\cite{AI:2015}. These products were also studied using free probability theory,
see Ref.~\cite{Burda} for a review on this topic.
However, with the help of free probability, only the global spectral statistics 
in the limit of large matrix dimensions are accessible.
In the present work we consider the statistics at finite matrix dimension,
which also allow investigations of the local spectral statistics.
 
 All ensembles mentioned above have two properties in common. First they are bi-unitarily invariant~\cite{HP-book,GKT:2014} (also known as isotropic~\cite{isotropic} or unitary rotation invariant~\cite{VKG-SUSY}) meaning for the weight $dP(g)$ on the complex square matrices $\mathbb{C}^{n\times n}$ we have
 \begin{equation}\label{bi-unitary}
 dP(g)=dP(k_1gk_2) \quad {\rm for\ all\ }k_1,k_2\in K \,,
 \end{equation}
with $K={\rm U}(n)$ the unitary group. The second property is that the joint probability density of their eigenvalues and their squared singular values have the particular form~\cite{AI:2015}
\begin{align}
\label{eq:EVD}
\FEV(z) \propto |\Delta_n(z)|^2 \prod_{j=1}^{n} \omega(|z_j|^2) \,, \qquad z \in \mathbb{C}^n
\end{align}
and
\begin{align}
\label{eq:SVD}
\FSV(a) \propto \Delta_n(a)  \det\left[(a_k \partial_{a_k})^{j-1}\omega(a_k)\right]_{j,k=1,\hdots,n} \,, \qquad \lambda \in \mathbb{R}_+^n
\end{align}
with the Vandermonde determinant $\Delta_n(a)=\prod_{1\leq b<c\leq n}(a_c-a_b)$ and $\omega: \mathbb{R}_+\rightarrow\mathbb{R}_+$ the corresponding one-point weight. For the induced Ginibre ensemble the one point-weight is $\omega(x)\propto x^\nu e^{-x}$ while for the induced Jacobi ensemble it is $\omega(x)\propto x^\nu (1-x)^\mu\Theta(1-x)$ and $\Theta$ is the Heaviside function. For products of random matrices it was found~\cite{AI:2015} that $\omega$ becomes a Meijer G-function. We call such an ensemble a Meijer G-ensemble which is part of a larger class called polynomial ensembles, see Definition~\ref{def:PE}.

Let us point out a very tricky issue with general ensembles satisfying Eqs.~\eqref{eq:EVD} and~\eqref{eq:SVD}. When choosing the one-point weight $\omega$ to be non-negative, the joint density~\eqref{eq:EVD} of the eigenvalues immediately becomes a joint probability density after proper normalization. Alas, this does not hold for the joint density~\eqref{eq:SVD} of the singular values which might be still a signed density. Therefore one has to be careful with a probabilistic interpretation for an arbitrary positive weight $\omega$. Only for a certain set of weights this interpretation is valid. However this issue will not be discussed in the present work since its main focus relies on the relation between the joint densities, regardless of whether they are signed or not. Indeed there are some applications in physics which also involve signed densities, e.g. QCD at finite chemical potential~\cite{Bloch-Wettig} and three-dimensional QCD with dynamical quarks~\cite{VerZah}.

A direct relation between the two joint densities as in Eqs.~\eqref{eq:EVD} and~\eqref{eq:SVD} is quite appealing and to establish this will be one of our main results in the present work. We aim at an even stronger statement. Namely, assuming that either the joint density of the squared singular values or of the eigenvalues of the random matrix $g$ is given and we know that its matrix weight is bi-unitarily invariant then we can give a closed formula for the other joint density in terms of a transformation of the other one. In particular we derive an explicit expression of a bijective operator $\mathcal{R}$ and its inverse such that
\begin{equation}\label{rel-SEV}
\mathcal{R}\FSV=\FEV\ {\rm and}\ \mathcal{R}^{-1}\FEV=\FSV.
\end{equation} 
We call $\mathcal{R}$ the \emph{SEV-transform} and establish the relation~\eqref{rel-SEV}
in Theorem~\ref{thm:Map} with the help of well-known harmonic analysis results on matrix spaces, 
see Refs.~\cite{Helgason3,FK,JL:SL2R,JL:SL2C,Terras}.

For the derivation of these results, it will be important that the \emph{spherical func\-tions}
associated with the general linear group $G=\GL(n,\mycmplx)$ have an explicit re\-presentation,
see Eq.~\eqref{eq:gelfand-naimark} below.
This will be used instead of the \emph{Harish-Chandra-Itzykson-Zuber integral}~\cite{HC,IZ} 
or related group integrals~\cite{group-int-1,group-int-2,KKS:2015}
which usually play a key role in the derivation of the joint density \eqref{eq:SVD} of the squared singular values.
Moreover, the explicit representation of the spherical functions also leads
to a fairly explicit relation 
between the joint densities of the squared singular values and the eigenvalues
of bi-unitarily invariant random matrices, regardless of whether those densities
are positive or signed, see Theorems \ref{thm:Map} and \ref{thm:pol-SEV}.

Explicit relations between eigenvalues and singular values have also useful applications. For example 
they can be found in QCD at non-zero chemical potential~\cite{KWY} and in wireless telecommunications~\cite{MIMO}.

As a direct application of our results described above, we derive explicit relations between the kernels for the eigenvalue and the squared singular value statistics in the case of an ensemble whose eigenvalues and squared singular values satisfy the joint densities~\eqref{eq:EVD} and \eqref{eq:SVD}, respectively. These ensembles are a particular kind of polynomial ensembles~\cite{KZ:2014,KS:2014} which we call \emph{polynomial ensembles of derivative type}. The limit of large matrix dimensions of those kernels and, thus, universality questions are not addressed. The discussion of the relation between the kernels shall only illustrate how useful the new approach is to get a deeper insight into different spectral statistics and their relations.

Another application illustrating our main result is the spectral statistics of deformations of the ensemble which break the bi-unitary invariance in a specific way.
% We understand by the notion ``breaking" that the deformation is not bi-unitarily invariant anymore. 
One of these deformations is in the form $\exp[{\rm Re}\,(\alpha\tr X)]$ recently employed as the coupling in the discussion of a product of two coupled matrices~\cite{AS-coupled-1,AS-coupled-2}. Another deformation has the form $|\det(\alpha\eins_n-g)|^\gamma$ which may open a hole at $\alpha$ in the spectrum if $\gamma$ approaches infinity. These two deformations show that our approach is by far not restricted to bi-unitarily invariant ensembles.

The present work is organized as follows.
In Section~\ref{sec:Preliminaries}, we introduce our notation
and recall the definitions and the basic properties
of the Mellin transform, the Harish transform and the spherical transform. 
Since we often need variations of results from the literature,
\eg adaptions to the set of functions we are considering, 
we provide proofs for several statements.
This will also lead to a self-contained presentation of our work,
and it will give the reader an insight into the main ideas of our results. 
Those main results and explicit formulas are presented in Sections~\ref{sec:MainResults}~and~\ref{sec:Implications}. 
In particular, the mapping~\eqref{rel-SEV} is derived in Theorem~\ref{thm:Map}.
In Section~\ref{sec:MainResults} we also prove that the relation between Eqs.~\eqref{eq:EVD} and \eqref{eq:SVD} holds for polynomial ensembles of derivative type, see Theorem~\ref{thm:pol-SEV}. Additionally we generalize this relation to certain deformations breaking the bi-unitary invariance in Theorem~\ref{thm:deform} and Corollary~\ref{cor:deform}. In Section~\ref{sec:Implications} we discuss the relation between the spectral statistics of the singular values and the eigenvalues of polynomial ensembles of derivative type in detail. In particular we derive simple, explicit relations between the kernels and the polynomials of the singular value and eigenvalue statistics, see Theorem~\ref{thm:rel-pol-SEV}. In Section~\ref{sec:Conclusions} we briefly discuss our results and give an outlook what questions are still open and should be addressed in future investigations.

 Let us again emphasize that we do not consider any limit of large matrix dimensions or other approximations. All of our results are exact and for finite matrix dimension.

% ****************************************************************************************************

\section{Preliminaries}
\label{sec:Preliminaries}

In this section, we introduce our notation and recall a number of known results
from random matrix theory and harmonic analysis on matrix spaces. 
Especially, we define the matrix spaces and the set of densities we make use of 
in Subsection~\ref{subsec:MatrixSpaces}. After this quite technical introduction we recall the Mellin transform in Subsection~\ref{subsec:MellinTransform}, the Harish transform in Subsection~\ref{subsec:HarishTransform} and the spherical transform in Subsection~\ref{subsec:SphericalTransform}. In those subsections we prove certain lemmas which are adaptions of known results to the set of densities we are considering.

% ****************************************************************************************************

\subsection{Matrix Spaces \& Function Sets}\label{subsec:MatrixSpaces}

In the sequel, we consider densities on various matrix spaces. 
We write $x_{jk}$ for the entries of a general matrix $x$
and $x_j$ for the diagonal entries of a diagonal matrix $x$.
The relevant matrix spaces are listed in Table~\ref{table:matrixspaces}.
For the unitary group $K={\rm U}(n)$, we choose the normalized Haar measure
such that $\int_K d^*k = 1$. For the sake of clarity
(and since we deviate from some parts of the cited literature here),
let us emphasize that for the general linear group $G=\GL(n,\mathbb{C})$, the measure $dg$ does \emph{not} denote the integration 
with respect to the Haar measure, which would be $d^* g := |\det g|^{-2n} \, dg$ in our notation.
Similar remarks apply to the groups of positive diagonal matrices $A$ and of complex diagonal entries $Z$.
For these groups, we use the natural isomorphisms $A \cong \myreal_+^n$ and $Z \cong \mycmplx_*^n$, 
with $\mathbb{R}_+$ the positive real axis without the origin and $\mathbb{C}_*=\mathbb{C}\setminus \{0\}$.

\begin{table}
$$
\begin{array}{|l|l|l|}
\hline
\text{Matrix Space}   & \text{Description} & \text{Reference Measure} \\
\hline\hline
\Omega = \Pos(n,\mycmplx)                     & \text{the cone of positive-definite}  & \\
       \quad\cong\GL(n,\mathbb{C})/{\rm U}(n)                & \quad \text{Hermitian matrices} & dy = \prod_{j \leq k} dy_{jk}  \\
  \hline
G = \GL(n,\mycmplx)   & \text{the general linear group}                & dg = \prod_{jk} dg_{jk} \\
\hline
K = {\rm U}(n)              & \text{the unitary group}                       & d^*k = \text{(normalized)} \\
                      &                                                & \qquad\ \text{Haar measure} \\
\hline
A =[\GL(1,\mathbb{C})/{\rm U}(1)]^n & \text{the group of positive definite}          & \\
 \quad\cong \mathbb{R}_{+}^{n}                     & \quad \text{diagonal matrices}   & da = \prod_j da_j \\
 \hline
T                & \text{the group of upper}  & \\
                      & \quad \text{unitriangular matrices}      & dt = \prod_{j < k} dt_{jk} \\
\hline
Z =[\GL(1,\mycmplx)]^n & \text{the group of invertible}          & \\
  \quad\cong \mathbb{C}_{*}^{n}          & \quad \text{complex diagonal matrices} & dz = \prod_j dz_j \\
\hline
\end{array}
$$
\caption{Matrix Spaces and Reference Measures. 
Here $dx$ denotes the Lebesgue measure on $\myreal$ if $x$ is a real variable
and the Lebesgue measure $dx = d\re x \, d\im x$ on $\mycmplx$ if $x$ is a complex variable.}
\label{table:matrixspaces}
\end{table}

By a \emph{density} on a matrix space, we understand a Borel-measurable function
which is Lebesgue integrable with respect to the corresponding reference measure.
Note that, unless otherwise indicated, we do not assume a density to be non-negative.
Sometimes (but not always) we denote those densities which are not non-negative  
by \emph{signed densities} to emphasize this point.
When a density is non-negative, we call it a \emph{positive density},
or a \emph{probability density} if it is additionally normalized.

We will always assume that the (possibly signed) measures under consideration have densities.
Thus, in particular, we exclude point measures like Dirac delta distributions. 
We would expect that our results can be extended to such distributions as well.
However including such distributions would make the presentation even more technical 
and the ideas and the approach we are pursuing less transparent.

Let us recall that given a signed measure on $G = \GL(n,\mycmplx)$ with a density,
the induced measures of the singular values and of the eigenvalues also have densities.
Especially, the singular values and the eigenvalues are pairwise different
apart from a null set. In the subsequent analysis, the exceptional null~sets 
where this is not the case are ignored without further notice. 
Also, the considered matrices are diagonalizable apart from a null set, 
and we do not run into trouble because matrix decompositions become singular. 

Given a signed density $f_G(g)$ on $G = \GL(n,\mycmplx)$, we typically write $\FSV(\lambda)$
for the induced joint density of the squared singular values $a={\rm diag}(a_1,\ldots,a_n)$
and $\FEV(z)$ for the induced joint density of the eigenvalues  $z={\rm diag}(z_1,\ldots,z_n)$.
Unless otherwise indicated, we assume that these densities are \emph{symmetric}
in their arguments, i.e.\@ invariant with respect to permutations. 
The resulting densities are obtained via the squared singular value decomposition $g=k_1\sqrt{a} k_2$ 
with $(k_1,k_2)\in[{\rm U}(n)\times {\rm U}(n)]/[{\rm U}(1)]^n$ 
and the eigendecomposition $g=hzh^{-1}$ with $h\in\GL(n,\mathbb{C})/[\GL(1,\mathbb{C})]^n$,
for almost every $g\in G$.
More precisely, we~have the explicit relations~\cite{Ginibre}
\begin{equation}\label{rel-sv}
\FSV(a)=\left(\frac{\pi^{n^2}}{n!}\prod_{j=0}^{n-1}\frac{1}{(j!)^2}\right)|\Delta_n(a)|^2\int_{[{\rm U}(n)\times {\rm U}(n)]/[{\rm U}(1)]^n} f_G(k_1\sqrt{a} k_2) d^*(k_1,k_2)
\end{equation}
and
\begin{equation}\label{rel-ev}
\FEV(z)= \frac{1}{n!} |\Delta_n(z)|^4\int_{\GL(n,\mathbb{C})/[\GL(1,\mathbb{C})]^n} f_G(hzh^{-1})d^*h
\end{equation}
with $d^*(k_1,k_2)$ and $d^*h$ the measures on the corresponding cosets induced by the Haar measures 
of the groups $K \times K = {\rm U}(n)\times {\rm U}(n)$ and $G=\GL(n,\mathbb{C})$.
The measure on $[{\rm U}(n)\times {\rm U}(n)]/[{\rm U}(1)]^n$ is normalized,
while that on $\GL(n,\mathbb{C})/[\GL(1,\mathbb{C})]^n$ is not, the set being non-compact.
%The measures readily result from the semi-Riemannian invariant length element on $G$ corresponding to the Haar-measure $d^*g=|\det g|^{-2n}dg$, namely via~\cite{}\alert{Referenz}
%\begin{eqnarray}
%{\rm Re}\,\tr dg dg^{-1}&=&-\tr\left(\frac{da}{2a}\right)^2\nonumber\\
%&&\hspace*{-0.3cm}-{\rm Re}\,\tr[(k_1^{-1}dk_1a^{1/2}+a^{1/2}dk_2k_2^{-1})(a^{-1/2}k_1^{-1}dk_1+dk_2k_2^{-1}a^{-1/2})]\nonumber\\
%&=&-{\rm Re}\,\tr \left(\frac{dz}{z}\right)^2+{\rm Re}\,\tr [h^{-1}dh,z]_-[h^{-1}dh,z^{-1}]_-. \label{inv-length}
%\end{eqnarray}
%The commutator is denoted by $[.,.]_-$. Moreover 
The Vandermonde determinant to the quartic power in Eq.~\eqref{rel-ev} is due to the fact that we have twice as many degrees of freedom in $\GL(n,\mathbb{C})/[\GL(1,\mathbb{C})]^n$ than in ${\rm U}(n)/[{\rm U}(1)]^n$. However it does not imply that the level repulsion is of order four,
since the coset integral diverges as $z$ becomes degenerate.
Indeed one can easily show that in the vicinity of two almost degenerate eigenvalues $z_1\approx z_2$ 
the integral diverges as $1/|z_1-z_2|^2$ when the function $f_G$ is smooth. 
The square reflects the number of non-compact directions 
of the coset $\GL(2,\mathbb{C})/[\GL(1,\mathbb{C})\times\GL(1,\mathbb{C})]$.

Additionally, we will need the following matrix factorizations and their transformations of measures.
Here we prefer to state the changes of measures in terms of the unitary group $K={\rm U}(n)$
instead of the coset space ${\rm U}(n)/[{\rm U}(1)]^n$.

\begin{remark}[Matrix Decompositions]
\label{rem:factorizations} \ 

(i) The \emph{Cholesky decomposition}~\cite[Theorem 2.1.9]{Muirhead} states that every $y \in \Omega$
has a representation $y = t^* a t$, where $a \in A$ and $t \in T$ are unique.
The associated change of measure is
\begin{equation}\label{Chol-dec}
dy = \left(\prod_{j=1}^{n} a_j^{2(n-j)}\right) \, da \, dt.
\end{equation}

(ii) The \emph{spectral decomposition}~\cite[Chapter 3.3]{Hua} states that every $y \in \Omega$
has a representation $y = k^* a k$, where $a \in A$ and $k \in K$. Let us recall that this decomposition becomes unique when replacing $K$ by the coset space ${\rm U}(n)/[{\rm U}(1)]^n$ and ordering the eigenvalues $a$.
The associated change of measure is
\begin{equation}\label{spec-dec}
dy = \left(\frac{1}{n!}\prod_{j=0}^{n-1}\frac{\pi^{j}}{j!}\right) |\Delta_n(a)|^2 \, da \, d^*k,
\end{equation}
Here it is worth mentioning that if $f_G$ is a density on $G$ and $f_\Omega$ is its induced density on $\Omega$
under the mapping $g \mapsto g^* g$, we have the relation
\begin{equation}\label{G-om-rel}
f_\Omega(y) = \left(\prod_{j=0}^{n-1}\frac{\pi^{j+1}}{j!}\right) \int_K f_G(k \sqrt{y}) \, d^*k \,.
\end{equation}
Then the combination of Eq.~\eqref{spec-dec} and Eq.~\eqref{G-om-rel} can be interpreted
as the transformation formula for the squared singular value decomposition $g=k_1 \sqrt{a} k_2$,
cf. Eq.~\eqref{rel-sv}.

(iii) The \emph{Schur decomposition} \cite{Schur}, \cite[Chapter 3.4]{Hua} states that every $g \in G$
has a representation $g = k^* z t k$, where $z \in Z$, $t \in T$, and $k \in K$.
Again, this decomposition becomes unique when replacing $K$ by the coset space 
${\rm U}(n)/[{\rm U}(1)]^n$ and ordering the eigenvalues of $g$.
The associated change of measure is
\begin{equation}\label{schur-dec}
dg = \left(\frac{1}{n!}\prod_{j=0}^{n-1}\frac{\pi^{j}}{j!}\right) |\Delta_n(z)|^2 
	\left( \prod_{j=1}^{n} |z_j|^{2(n-j)} \right) \, dt \, dz \, d^*k,
\end{equation}
with the same normalization constant as in Eq.~\eqref{spec-dec}. Let us mention that we employ a Schur decomposition which is slightly different from the one usually used in the literature. The other one is additive in $Z$ and $T$, i.e. $g = k^* (z+ t'-\eins_n) k$ with $t'\in T$ and $\eins_n$ the $n$-dimensional identity matrix. One can readily show that both transformations are equivalent due to the substitution $t'-\eins_n=z (t-\eins_n)$. Moreover, the Schur decomposition is  practically more convenient than the eigendecomposition in the eigenvalues $z$, see Eq.~\eqref{rel-ev}, despite the fact that with both decompositions we get the same joint density of the eigenvalues. The counterpart of Eq.~\eqref{rel-ev} for the Schur decomposition reads 
\begin{equation}\label{rel-ev-2}
\FEV(z) =\left(\frac{1}{n!}\prod_{j=0}^{n-1}\frac{\pi^{j}}{j!}\right) |\Delta_n(z)|^2\left(\prod_{j=1}^{n} |z_j|^{2(n-j)}\right) \int_{T} \int_K f_G(k^* z t k) \, d^*k \, dt.
\end{equation}
The order of the integrals do not play a role when the density $f_G$ is Lebesgue integrable, i.e. $f_G\in L^1(G)$. Moreover the integral over the non-compact group $T$ typically remains finite even in the limit of a degenerate spectrum.
This is also reflected in the different power of the Vandermonde determinant, which is consistent with
the discussion in the paragraph below Eq.~\eqref{rel-ev}. 
\end{remark}

Before going over to the other ingredients of our analysis, we recall that the group $G$ has a natural action on $\Omega=\Pos(n,\mycmplx)$, namely via $y\mapsto g^{*} y g$.
More precisely, this group action is transitive,
and the isotropy group of the identity matrix is the unitary group $K$. Thus, 
we~have a~bijection 
\begin{align}
\label{eq:bijection}
\begin{array}{ccc}
K \backslash G &\overset{\cong}{\longrightarrow}& \Pos(n,\mycmplx) \,, \\
Kg             & \longmapsto                    & g^*g             \,, \\
\end{array}
\end{align}
where $K \backslash G$ denotes the space of right cosets of $G$ with respect to $K$.

We call a function $f_\Omega$ on $\Omega$ \emph{$K$-invariant}
if $f_\Omega(k^*yk) = f_\Omega(y)$ for all $k \in K$ and $y \in \Omega$.
Also, we call a function $f_G$ on $G$ 

(i) \emph{$K$-left-invariant} if $f_G(kx) = f_G(x)$ for all $k \in K$, $x \in G$,

(ii) \emph{$K$-right-invariant} if $f_G(xk) = f_G(x)$ for all $k \in K$, $x \in G$,

(iii) \emph{bi-unitarily invariant} if $f_G$ is $K$-left-invariant and $K$-right-invariant.

\noindent
Similar terminology will be used for probability measures and for random matrices,
where it means that the ensemble has the respective invariance properties.

To derive the relation between the joint densities of the singular values and eigenvalues,
we need to introduce several sets of functions as well as mappings between them;
see Eq.~\eqref{diagram} for the final diagram which we will show to be commutative. 
This diagram will be at the heart of our approach and of the proof 
for the relation between the joint densities of the singular values and the eigenvalues.

First of all, let us identify $K$-left-invariant functions on $G$
with functions on $K \backslash G$. Then, under the bijection \eqref{eq:bijection},
the $K$-invariant functions on $\Omega$ correspond to
the bi-unitarily invariant functions on $G$. 
Furthermore, using the spectral decomposition $y = k^* a k$
(where $y \in \Omega$, $k \in K$, $a \in A$),
it is clear that any $K$-invariant function on~$\Omega$ 
may be regarded as a symmetric function on~$A$.
Thus, we obtain a correspondence for functions:
\begin{align}
\begin{array}{ccccc}
\begin{array}{c} \text{bi-unitarily invariant}\\ \text{functions} \\ \text{on $G$} \end{array}&
\longleftrightarrow &
\begin{array}{c} \text{$K$-invariant} \\ \text{functions} \\ \text{on $\Omega$} \end{array}&
\longleftrightarrow &
\begin{array}{c} \text{symmetric} \\ \text{functions} \\ \text{on $A$} \end{array}
\end{array}
\label{eq:correspondence-1}
\end{align}
More explicitly, given a symmetric function $F_{\rm SV}$ on $A$,
the corresponding $K$-invariant function $F_\Omega$ on $\Omega$ 
is given by $F_\Omega(y) :=F_{\rm SV}(\lambda(y))$,
where $\lambda(y)$ denotes the diagonal matrix in the spectral decomposition of $y$,
and the corresponding bi-unitarily invariant function $F_G$ on $G$ 
is given by $F_G(g) :=F_\Omega(g^* g)$.

A similar correspondence holds at the level of densities (signed, non-negative or even probability densities)
and their induced densities:
\begin{align}
\begin{array}{ccccc}
L^{1,K}(G) &
\underset{\displaystyle\mathcal{I}_{\Omega}^{-1}}{\overset{\displaystyle\mathcal{I}_{\Omega}}{\text{\scalebox{3}[1]{$\rightleftarrows$}}}} &  L^{1,K}(\Omega) & \underset{\displaystyle\mathcal{I}_{A}^{-1}}{\overset{\displaystyle\mathcal{I}_{A}}{\text{\scalebox{3}[1]{$\rightleftarrows$}}}} & L^{1,{\rm SV}}(A)
\end{array}
\label{eq:correspondence-2}
\end{align}
The set $L^{1,K}(G)$ comprises the bi-unitarily invariant densities on the complex square matrices,
while $ L^{1,K}(\Omega)$ and $L^{1,{\rm SV}}(A)$ are the sets of all invariant densities 
on the positive definite Hermitian matrices and of all symmetric densities 
on the squared singular values (aligned as diagonal matrices in $A$), respectively. 
More precisely, the sets are defined as
\begin{equation}
L^{1,K}(G):=\{f_G\in L^1(G) \,|\, \text{$f_G$ is a bi-unitarily invariant density on }G\}, \label{def-L1G}
 \end{equation}
 and
 \begin{equation} \label{def-L1Om-A}
 L^{1,K}(\Omega):=\mathcal{I}_\Omega L^{1,K}(G)\ {\rm and}\ L^{1,{\rm SV}}(A):=\mathcal{I}_A L^{1,K}(\Omega).
 \end{equation}
Note that $L^{1,{\rm SV}}(A)$ is simply the set of all symmetric $L^1$-functions on $\mathbb{R}_+^n$.
Let us recall that the measures occurring in this paper will have densities 
with respect to the reference measures in Table \ref{table:matrixspaces}. 

The maps appearing in Eqs.~\eqref{eq:correspondence-2} and \eqref{def-L1Om-A} are explicitly given as follows.
The symmetric density $\FSV$ on $A$ corresponding to a given $K$-invariant density $f_\Omega$ on $\Omega$ 
is
\begin{eqnarray}\label{I-A-def}
 \mathcal{I}_A f_\Omega(a)&:=&\left(\frac{1}{n!}\prod_{j=0}^{n-1}\frac{\pi^j}{j!}\right)|\Delta_n(a)|^2f_\Omega(a)=\FSV(a),\\
   \mathcal{I}_A^{-1}\FSV(y)&= &\left(n!\prod_{j=0}^{n-1}\frac{j!}{\pi^j}\right)\frac{\FSV(\lambda(y))}{ |\Delta_n(\lambda(y))|^2}=f_\Omega(y)\nonumber
\end{eqnarray}
 with the linear operator $\mathcal{I}_A:L^{1,K}(\Omega)\rightarrow L^{1,{\rm SV}}(A)$. 
 This operator is bijective because $\lambda(a)=a$ and $f_\Omega(y)=f_\Omega(\lambda(y))$
 for any $f_\Omega\in L^{1,K}(\Omega)$. The normalization constant comes from Eq.~\eqref{spec-dec}. 
 The $K$-invariant density $f_\Omega$ on $\Omega$ corresponding to a bi-unitarily invariant density $f_G(x)$ on $G$ 
is given by
\begin{eqnarray}\label{I-Om-def}
 \mathcal{I}_\Omega f_G(y)&:=&\left(\prod_{j=0}^{n-1}\frac{\pi^{j+1}}{j!}\right)f_G(\sqrt{y})=f_\Omega(y),\\
   \mathcal{I}_{\Omega}^{-1}f_\Omega(g)&= &\left(\prod_{j=0}^{n-1}\frac{j!}{\pi^{j+1}}\right)f_\Omega(g^*g)=f_G(g)\nonumber
\end{eqnarray}
with $\mathcal{I}_\Omega:L^{1,K}(G)\rightarrow L^{1,K}(\Omega)$ and the normalization constant of Eq.~\eqref{G-om-rel}. The bijectivity of the map $\mathcal{I}_\Omega$ follows from $f_G(g)=f_G(\sqrt{g^*g})=f_G(\sqrt{gg^*})$ for all $f_G\in L^{1,K}(G)$ due to bi-unitary invariance.
We underline that the correspondence \eqref{eq:correspondence-2} for densities 
is a bit different from that \eqref{eq:correspondence-1} for functions
due to their different transformation properties under changes of coordinates.

\begin{remark}[Crux of Bi-unitarily Invariant Densities]
\label{rem:specification}\

Summarizing the discussion above, in order to specify 
a $K$-invariant measure on $\Omega=\Pos(n,\mycmplx)$ 
or a bi-unitarily invariant measure on $G=\GL(n,\mycmplx)$,
we will usually specify the symmetric  density $\FSV(a)$ on $A$
and use the correspondence \eqref{eq:correspondence-2}. 
Thus we once again underline that this is a one-to-one correspondence because the normalized Haar measure of the unitary group $K$ and, hence, the induced measure on the coset $[{\rm U}(n)\times{\rm U}(n)]/[{\rm U}(1)]^n$ is unique. Assuming we know the joint density $\FSV$ of the singular values $\lambda$ of $g\in G$ and that the density $f_G$ is bi-unitarily invariant we know the whole measure on $G$ via $f_G(g)dg=\widehat{c}\FSV(\lambda)d\lambda d^*(k_1,k_2)$ with $(k_1,k_2)\in[{\rm U}(n)\times{\rm U}(n)]/[{\rm U}(1)]^n$ and $\widehat{c}$ the normalization constant in front of the integral~\eqref{rel-sv}.
\end{remark}

The bi-unitary invariance of densities on $G$ gives rise to another non-trivial relation between densities on $A$ and on $Z$. This relation becomes useful when studying the eigenvalues of a matrix $g\in G$. To this end, we define a set of functions on $Z$
\begin{equation}\label{def-L1Z}
 L^{1,{\rm EV}}(Z):=\mathcal{T} L^{1,K}(G),
\end{equation}
where
\begin{multline}
\mathcal{T} : L^1(G)\rightarrow L^1(Z), \ \mathcal{T}f_G(z) := \left(\frac{1}{n!}\prod_{j=0}^{n-1}\frac{\pi^{j}}{j!}\right)|\Delta_n(z)|^2\left(\prod_{j=1}^n|z_j|^{2(n-j)}\right) \\
\times\int_{T}\left(\int_{K} f_G(k^* z t k) d^*k\right) dt=\FEV(z) \,,\ \label{T-def}
\end{multline}
Elements in $ L^{1,{\rm EV}}(Z)$ are the induced joint densities of eigenvalues of bi-unitarily invariant matrix ensembles.
Another set of functions on $A$ related to $ L^{1,{\rm EV}}(Z)$ is
\begin{multline}\label{def-L1A-new}
 L^{\mathcal{H}}(A) := \big\{ f : A \to \myreal \,\big|\, \text{$f$ measurable} \\ 
 	\text{ and } |\Delta_n(z)|^2 (\det|z|)^{n-1} f(|z|^2) \in L^{1,{\rm EV}}(Z) \big\},
\end{multline}
Strictly speaking, we do not consider functions, but equivalence classes of functions
(similarly as for the $L^1$-spaces), where two functions $f_1$ and $f_2$ are regarded 
as equi\-valent when $f_1 = f_2$ almost everywhere.
The name $L^{\mathcal{H}}(A)$ is due to the fact that this set will turn out to be the image
of the Harish transform to be introduced in Subsection \ref{subsec:HarishTransform} below.
Furthermore, let us mention that the functions in $L^{\mathcal{H}}(A)$
are closely related to the joint densities of the \emph{radii} of the eigenvalues;
see Remark~\ref{rem:RadiusDistribution} below.

{It is worth emphasizing that, contrary to what Definition \ref{T-def} might suggest,
functions in $L^{1,{\rm EV}}(Z)$, and hence in $L^{\mathcal{H}}(A)$, are symmetric 
in their arguments.

The relation between $L^{\mathcal{H}}(A)$ and $L^{1,{\rm EV}}(Z)$ is given by the bijective map $\mathcal{I}_Z: L^{\mathcal{H}}(A)\rightarrow L^{1,{\rm EV}}(Z) $ with
\begin{eqnarray}\label{I-Z-def}
 \mathcal{I}_Z f_A(z)&:=&\frac{1}{n!\pi^n}|\Delta_n(z)|^2(\det |z|)^{n-1}f_A(|z|^2)=\FEV(z),\quad\\
 \mathcal{I}_Z^{-1}\FEV(a)&=&n!\pi^n \frac{\FEV(\sqrt{a})}{|\Delta_n(\sqrt{a})|^{2} \, (\det a)^{(n-1)/2}}
=f_A(a).\nonumber
\end{eqnarray}
These maps are obviously well-defined.

\begin{remark}[Eigenvalue Radius Distribution]
\label{rem:RadiusDistribution}
The inverse of the map $\mathcal{I}_Z$ has another interesting representation,
viz.
\begin{align}
\label{I-Z-alt}
\mathcal{I}_Z^{-1}\FEV(a)=n!\frac{\pi^n \int_{[{\rm U}(1)]^n}\FEV(\sqrt{a}\Phi)d^*\Phi}{ (\det a)^{(n-1)/2}{\rm Perm}[a_b^{c-1}]_{b,c=1,\ldots,n}} =f_A(a) \,.
\end{align}
Here, ``${\rm Perm}$'' is the permanent, and the diagonal matrix $\Phi={\rm diag}(e^{\imath\varphi_1},\ldots,e^{\imath\varphi_n})\in[{\rm U}(1)]^n$ is distributed via the normalized Haar measure $d^*\Phi=(2\pi)^{-n}d\varphi_1\cdots d\varphi_n$ on $[{\rm U}(1)]^n$. 
To prove Eq.~\eqref{I-Z-alt}, we need two observations. First,
\begin{align}
\label{eq:VandermondeIntegral}
\int_{[{\rm U}(1)]^n} |\Delta_n(\sqrt{a} \Phi)|^2 \, d^*\Phi = {\rm Perm}[a_b^{c-1}]_{b,c=1,\hdots,n} \,,
\end{align}
as is readily verified by expanding the Vandermonde determinant 
$\Delta_n(\sqrt{a} \Phi) = \det[(\sqrt{a_b} e^{\imath\varphi_b})^{c-1}]_{b,c=1,\hdots,n}$ 
with the aid of the Leibniz formula and by analyzing which of the resulting terms remain left
after the integration over $\Phi$.
Second, any density $\FEV \in L^{1,\rm EV}(Z)$ has the form
$\FEV(z) = g(z) |\Delta(z)|^2$, where $g(z)$ satisfies $g(a \Phi) = g(\Phi a) = g(a)$
for any $a \in A$, $\Phi \in [{\rm U}(1)]^n$.
Combining these observations, it follows that
\begin{align}
   \int_{[{\rm U}(1)]^n}\FEV(\sqrt{a}\Phi)d^*\Phi 
&= g(\sqrt{a}) {\rm Perm}[a_b^{c-1}]_{b,c=1,\hdots,n} \nonumber \\
&= \frac{\FEV(\sqrt{a})}{|\Delta_n(\sqrt{a})|^2} {\rm Perm}[a_b^{c-1}]_{b,c=1,\hdots,n} \,,
\end{align}
whence Eq.~\eqref{I-Z-alt}.  

Finally, let us note that if $\FEV$
is the joint density of the eigenvalues, the numerator in Eq.~\eqref{I-Z-alt}
is essentially the joint density of the squared eigenvalue radii.
In this respect, functions in $f^\mathcal{H}(A)$ are related
to the distributions of the eigenvalue radii
of bi-unitarily invariant random matrices.
\end{remark}

It is very important to remark that in general $L^{1,{\rm SV}}(A)\neq L^{\mathcal{H}}(A)$ despite the fact that they are both spaces of functions on $A$.

\begin{lemma}[Integrability Properties of Functions in $L^{\mathcal{H}}(A)$]\label{lem:sets}\

\noindent{}Let $f_A\in L^{\mathcal{H}}(A)$, then we have
\begin{equation}\label{finiteness}
\int_A \Bigl|{\rm Perm}[a_b^{\varrho'_c-1}]_{b,c=1,\ldots,n}f_A(a)\Bigr|da<\infty
\end{equation}
with $\varrho'=\varrho+n\eins_n$ where
\begin{align}
\label{eq:varrho}
\varrho := \diag(\varrho_1,\ldots,\varrho_n), \ {\rm with}\ \varrho_j:=\frac{2j-n-1}{2},\ j=1,\ldots,n.
\end{align}
\end{lemma}
The diagonal matrix $\varrho$ is essentially the sum of all positive roots of $K={\rm U}(n)$, see Ref.~\cite[Chapter IV.4]{Helgason3}, and will occur frequently in the next subsections, too.

\begin{proof}
 Let $f_A\in L^{\mathcal{H}}(A)$ and $\FEV:=\mathcal{I}_Z^{-1}f_A\in L^{1,{\rm EV}}(Z)$,
 and let $f_G\in L^{1,K}(G)$ be a~density with $\mathcal{T}f_G=\FEV$. 
 Consider the integral \eqref{finiteness}.
 After replacing $f_A(a)$ with Eq.~\eqref{I-Z-alt}, 
 the permanent in \eqref{finiteness} cancels with the determinant and the permanent 
 in Eq.~\eqref{I-Z-alt}, and we are left with the numerator in Eq.~\eqref{I-Z-alt}.
 But the latter must be integrable, being the density of the squared eigenvalue radii
 induced by the density $f_G \in L^{1,K}(G)$.
\end{proof}

Before closing this subsection, let us introduce generalizations of the operators $\mathcal{I}_\Omega$ and  $\mathcal{I}_A$, see Eqs.~\eqref{I-A-def} and \eqref{I-Om-def}, to densities which are not necessarily bi-unitarily invariant.  Those operators shall act on the sets $L^1(G)$ and $L^1(\Omega)$ as
\begin{eqnarray}
\mathcal{K}_{\Omega}\hspace*{-0.2cm}&:&\hspace*{-0.2cm}L^1(G)\rightarrow L^1(\Omega),\ \mathcal{K}_{\Omega}f_G(y):=\left(\prod_{j=0}^{n-1}\frac{\pi^{j+1}}{j!}\right)\int_K f_G(k\sqrt{y})dk,\label{K-Om-def}\\
\mathcal{K}_A\hspace*{-0.2cm}&:&\hspace*{-0.2cm}L^1(\Omega)\rightarrow L^1(A),\ \mathcal{K}_Af_\Omega(a):=\left(\frac{1}{n!}\prod_{j=0}^{n-1}\frac{\pi^{j}}{j!}\right) |\Delta_n(a)|^2\int_K f_\Omega(k^*ak)d^*k.\qquad\ \label{K-A-def}
\end{eqnarray}
Indeed, the functions in the ranges of these maps are obviously Lebesgue integrable by construction;
see Eqs.~\eqref{spec-dec} and \eqref{G-om-rel}.
With these definitions and Eq.~\eqref{T-def}, the relation between densities on $G$ and those on $A$ and $Z$, see Eqs.~\eqref{rel-sv} and \eqref{rel-ev-2}, compactly reads
\begin{equation}\label{rel-comp}
 \FSV=\mathcal{K}_A\mathcal{K}_{\Omega}f_G\ {\rm and}\ \FEV=\mathcal{T}f_G.
\end{equation}
 In the case that $f_G$ is bi-unitarily invariant we have the following simplification for the density of the squared singular values,
\begin{equation}\label{rel-comp-simple}
 \FSV=\mathcal{I}_A\mathcal{I}_{\Omega}f_G.
\end{equation}
Also the relation for $\FEV$ simplifies since the integral over the unitary group $K$ drops out. Let us emphasize 
that $\mathcal{I}_A\mathcal{I}_{\Omega}$ is invertible while $\mathcal{K}_A\mathcal{K}_{\Omega}$ is not;
the same is true for the operator $\mathcal{T}$ on the set $L^{1,K}(G)$ and on the set $L^1(G)$, respectively.
We need these generalizations for the applications to densities $f_G$ which break the bi-unitary invariance, 
as discussed in Subsection~\ref{subsec:break}.

Finally, for the discussion of the Mellin transform, see Subsection \ref{subsec:MellinTransform},
as well as for the polynomial ensembles of derivative type, see Definition~\ref{def:PE}, 
we also need a~particular subset of $L^1(\mathbb{R}_+)$, namely
\begin{eqnarray}
 L^{1,k}_{\mathbb{I}}(\mathbb{R}_+)&\mskip-9mu:=\mskip-9mu&\bigg\{f\in L^{1}(\mathbb{R}_+) \,\Big|\, 
 \text{$f$ is $k$-times differentiable} \label{func-def}\\
 && \text{and for all } \kappa \in\mathbb{I} \text{ and }j=0,\ldots k:
 \int_0^\infty \bigg|y^{\kappa-1}\Big({-}y\frac{\partial}{\partial y}\Big)^jf(y)\bigg|dy<\infty\bigg\}. \nonumber
\end{eqnarray}
Here, $\mathbb{I} \subset \mathbb{R}$ is an interval containing the number $1$,
so that all of the functions will be Lebesgue integrable. 
Whether this interval is open or not does not matter.
Furthermore, ``$k$-times differentiable'' means ``$(k-1)$-times continuously differentiable
and with a $(k-1)$th derivative which is absolutely continuous'',
and hence differentiable almost everywhere, see \eg \cite[Chapter~8]{Rudin} for details.
Finally, let $L^{1}_{\mathbb{I}} := L^{1,0}_{\mathbb{I}}$.

% ****************************************************************************************************

\subsection{Mellin Transform}\label{subsec:MellinTransform}

For a measurable function $f$ defined on $\mathbb{R}_+$, the \emph{Mellin transform} is defined by
\begin{equation}\label{M-def}
\mellin f(s) := \int_0^\infty f(y) \, y^{s-1} \, dy.
\end{equation}
It is only defined for those $s \in \mycmplx$ such that the integral exists (in the Lebesgue sense).
In particular, if $f \in L^1(\myreal_+)$, the Mellin transform is defined at least
on the line $1 + \i\myreal$.

\begin{remark}[Notation for Functionals]\label{rem:functional-notation}\ 

In the following we employ different but equivalent notations for functionals. For linear functionals as the Mellin transform we have the following equivalent notations
\begin{equation}\label{not-func}
\mellin f(s)=\mathcal{M}([f];s)=\mathcal{M}([f(y)];s).
\end{equation}
The last notation will especially be used when the argument $y$ of the function $f$ needs to be indicated.
Similar notations of last two kinds will be also employed for non-linear functionals like the normalization constants and the kernels of determinantal point processes, see Section~\ref{sec:Implications}.
\end{remark}

% Furthermore, let us state a particular form of the Mellin inversion formula which slightly differs from the literature, \eg see Ref.~\cite{JL:SL2R}, and becomes important for the present work. It is based on the fact that $\mathcal{M}$ is injective on  $L^1(\myreal_+)$.

Furthermore, let us state a particular version of the Mellin inversion formula 
which holds for general integrable functions, see also Theorem 1.32 in \cite{Titchmarsh}.
Since we will use similar arguments for the spherical inversion formula later,
we include an outline of the proof.

\begin{lemma}[Mellin Inversion on $L^{1}_{\mathbb{I}}(\mathbb{R}_+)$]\label{lem:Mel-Inv}\

\noindent{}The Mellin transformation $\mathcal{M}:L^{1}_{\mathbb{I}}(\mathbb{R}_+)\rightarrow \mathcal{M}L^{1}_{\mathbb{I}}(\mathbb{R}_+)$ on the set \eqref{func-def} is bijective with the \emph{Mellin inversion formula}
\begin{align}
\label{eq:MellinInversion}
\mellin^{-1}([\mellin f];x) := \lim_{\epsilon\to0} \int_{-\infty}^{\infty} 
	\frac{\pi^2\cos \epsilon s}{\pi^2-4\epsilon^2 s ^2}\mellin f(d+\imath s) \, x^{-d-\imath s} \, \frac{ds}{2\pi} = f(x)
% \mellin^{-1}([\mellin f];x) := \lim_{\epsilon\to0} \int_{-\infty}^{\infty} \frac{\pi^2\cos \epsilon (s-\imath d)} {\pi^2-4\epsilon^2 (s-\imath d)^2}\mellin f(d+\imath s) \, x^{-d-\imath s} \, \frac{ds}{2\pi} = f(x)
\end{align}
for any $d\in\mathbb{I}\subset\mathbb{R}$, $f\in L^{1}_{\mathbb{I}}(\mathbb{R}_+)$,
and almost all $x \in \mathbb{R}_+$. % The function ${\rm sign} (d)$ yields the sign of $d$.
\end{lemma}

Indeed we can always choose $\mathbb{I}=\{1\}$ meaning that $f\in L^{1}_{\{1\}}(\mathbb{R}_+)$ is only Lebesgue integrable. Therefore the formula~\eqref{eq:MellinInversion} also applies for any interval with $1\in\mathbb{I}$  since we have $L^{1}_{\mathbb{I}}(\mathbb{R}_+)\subset L^{1}_{\{1\}}(\mathbb{R}_+)$. The difficult part is the characterization of functions in $\mathcal{M}L^{1}_{\mathbb{I}}(\mathbb{R}_+)$. We do not completely address this issue here by only saying that $\mathcal{M}f$ is analytic on the strip $\mathbb{I}\times\imath\mathbb{R}$ when $\mathbb{I}$ is open, and it is even bounded when $\mathbb{I}$ is compact.

\begin{proof}
The first step is to replace $\epsilon s$ with $\epsilon (s-\imath d)$ in the large fraction under the integral. 
This is possible because
\begin{align}
\label{eq:mellin1}
\lim_{\epsilon \to 0} \int_{-\infty}^{\infty} \left| \frac{\pi^2\cos \epsilon s}{\pi^2-4\epsilon^2 s ^2} - \frac{\pi^2\cos \epsilon (s-\imath d)} {\pi^2-4\epsilon^2 (s-\imath d)^2} \right| \big| \mellin f(d+\imath s) \big| \, \big| x^{-d-\imath s} \big| \, \frac{ds}{2\pi} = 0
\end{align}
for any $x \in \myreal_{+}$. Eq.~\eqref{eq:mellin1} holds by dominated convergence,
since $\mellin f(d+\imath s)$ and $x^{-d-\imath s}$ are bounded
and $(\pi^2\cos \epsilon s)(\pi^2-4\epsilon^2 s^2)^{-1}$ is continuous and bounded
by ${\rm const} \, (1+|s|^2)^{-1}$ in any strip of bounded width around the real axis.

Now the inversion formula can be readily obtained by plugging the definition \eqref{M-def} 
into Eq.~\eqref{eq:MellinInversion} and interchanging the order of integration.
This is possible because the integrand (viewed as a bivariate function of $s$ and $y$)
is Lebesgue integrable. 
Then we have
\begin{eqnarray}
 &&\lim_{\epsilon\to0}\int_{0}^{\infty} \Theta\left[1-\frac{({\rm ln}(y/x))^2}{\epsilon^2}\right]\left[\left(\frac{y}{x}\right)^{\imath \pi/(2\epsilon)}+\left(\frac{y}{x}\right)^{-\imath \pi/(2\epsilon)}\right] f(y)  \frac{\pi dy}{4\epsilon y}\nonumber\\
 &=&\lim_{\epsilon\to0}\int_{-1}^{1}\cos\left(\frac{ \pi y'}{2}\right) f(x e^{\epsilon y'})  \frac{\pi dy'}{4}.\label{cal-2.1}
\end{eqnarray}
We now use a version of the Lebesgue differentiation theorem which states that
if~$f \in L^1(\myreal^n)$ and $B(0,1)$ denotes the unit ball in $\myreal^n$, 
we have
\begin{align}
\label{eq:DifferentiationTheorem}
\lim_{\epsilon \to 0} \int_{B(0,1)} |f(x+\epsilon y) - f(x)| \, dy = 0
\end{align}
for almost all $x \in \myreal^n$; \eg see \cite[Theorem 8.8]{Rudin}.
Of course, here we have $n = 1$, but we will need the multivariate formulation later on.

It is straightforward to show that if $x \in \myreal_{+}$ such that Eq.~\eqref{eq:DifferentiationTheorem} holds,
then the limit in Eq.~\eqref{cal-2.1} is equal to $f(x)$.
Indeed, since $\int_{-1}^{1} \cos(\pi y'/2) \, \pi dy'/4 = 1$, it is sufficient to show that
\begin{align}
\lim_{\varepsilon \to 0} \int_{-1}^{1}\cos\left(\frac{ \pi y'}{2}\right)
	\left( f(x e^{\epsilon y'}) - f(x) \right) \frac{\pi dy'}{4} = 0
\end{align}
for those $x\in\mathbb{R}_+$ such that Eq.~\eqref{eq:DifferentiationTheorem} holds.
This limit follows from the observation that the integral is bounded by
\begin{eqnarray}
\frac{\pi}{4} \int_{-1}^{1} \big|f(x e^{\epsilon y'}) - f(x)\big| \, dy'
&\leq& \frac{\pi}{4} \int_{-2x}^{+2x} \big|f(x + \varepsilon z) - f(x)\big| \, \frac{dz}{\epsilon z+x}\\
&\leq& \frac{\pi}{4(1-2\epsilon)x} \int_{-2x}^{+2x} \big|f(x + \varepsilon z) - f(x)\big| \, dz\nonumber
\end{eqnarray}
for all sufficiently small $\varepsilon > 0$. The reason for this estimate is that the original integration domain is $[xe^{-\epsilon},xe^{+\epsilon}]\subset[x(1-2\epsilon),x(1+2\epsilon)]$ for sufficiently small $\varepsilon > 0$.
This completes the proof of Lemma \ref{lem:Mel-Inv}.
\end{proof}

We would expect that the inversion formula can be generalized to distributions which are normalizable on $\mathbb{R}_+$ like the Dirac delta function. However, in practice the inversion formula~\eqref{eq:MellinInversion} is often even correct without the test function. For instance,
this is the case if the integrand is absolutely integrable, possibly after a suitable deformation of the contour, see \eg the definition of the Meijer G-function~\cite{Abramowitz} for an important example.

An important property of the Mellin transform is its action on a particular type of a differentiated function. 
Assuming that $f\in  L^{1,k}_{\mathbb{I}}(\mathbb{R}_+)$, we have
\begin{equation}
\label{eq:mellin-link}
\displaystyle\mellin  \big(\Big[ \Big({-}y \frac{\partial}{\partial y}\Big)^k f(y) \Big]; s \big)
= s^k \, \mellin f(s), \qquad \re s \in \mathbb{I} \,,
\end{equation}
as follows by integration by parts.

\medskip
The Mellin transformation can be readily extended to the multivariate functions in $L^{\mathcal{H}}(A)$ of $n$ positive real variables.
With a slight abuse of notation the multivariate Mellin transform $\mathcal{M}: L^{\mathcal{H}}(A)\rightarrow \mathcal{M}L^{\mathcal{H}}(A)$ is given by
\begin{equation}\label{mult-M-def}
 \mathcal{M}f_A(s):=\frac{1}{n!}\int_A f_A(a) {\rm Perm}[a_b^{s_c-1}]_{b,c=1,\ldots,n} da,
\end{equation}
where $f_A\in L^{\mathcal{H}}(A)$ and for those $s={\rm diag}(s_1,\ldots,s_n)\in \mathbb{C}^n$ such that the integrand is Lebesgue integrable. At least for $s=\varrho'$ this integral exists due to Lemma~\ref{lem:sets}. The inverse of the multivariate Mellin transformation is  a natural generalization of Eq.~\eqref{eq:MellinInversion} and reads
\begin{eqnarray}
\label{mult-MellinInversion}
\mellin^{-1}([\mellin f_A];a)
\hspace*{-0.3cm}&:=& \hspace*{-0.3cm} \frac{1}{n!}\lim_{\epsilon\to 0}\int_{\mathbb{R}^n} \zeta_1(\epsilon s)\mellin f_A(\varrho'+\imath s) {\rm Perm}[a_b^{-\varrho'_c-\imath s_c}]_{b,c=1,\ldots,n}\prod_{j=1}^{n}\frac{ds_j}{2\pi}\nonumber\\
 &=&\hspace*{-0.3cm} f(a)
\end{eqnarray}
with $f \in L^{\mathcal{H}}(A)$ and the regularizing function
\begin{equation}\label{zet-def}
\zeta_l(\epsilon s):=\prod_{j=1}^n \frac{\pi^{2l}\cos\epsilon s_j}{\prod_{k=1}^l(\pi^2-4\epsilon^2 s_j^2/(2k-1)^2)},\ l\in\mathbb{N}.
\end{equation}
The functions $\zeta_l$ with $l>1$ will be used for the inverse spherical transform, see Lemma~\ref{lem:sphe-inv} below.

We underline that $\mathcal{M}f$ is symmetric in its arguments by definition.
In~fact, since $f \in L^{\mathcal{H}}(A)$ is also symmetric in its arguments,
we could replace the permanent with the product $n! \, \prod_{j=1}^{n} a_j^{s_j-1}$
in~\eqref{mult-M-def}. However, we prefer to make the symmetry property more transparent
by employing the permanent in~\eqref{mult-M-def}, and similarly in~\eqref{mult-MellinInversion}.

This kind of generalization of the Mellin transform to multivariate functions also has an analogue on the matrix level, namely the spherical transform. The spherical transform will be discussed in Subsection~\ref{subsec:SphericalTransform}.

% ****************************************************************************************************

\subsection{Harish Transform}\label{subsec:HarishTransform}

In this and the next subsection, we introduce two transforms,
the \emph{spherical transform} and the \emph{Harish transform},
from harmonic analysis on matrix spaces which will be crucial for relating
the joint density of the eigenvalues to that of the singular values.
In doing so, we focus on selected results which will be needed later,
and refer to \textsc{Helgason} \cite{Helgason3} for a thorough introduction
to harmonic analysis on matrix spaces.
Let us also mention the monographs by \textsc{Terras} \cite{Terras},
\textsc{Faraut} and \textsc{Koranyi} \cite{FK}
as well as \textsc{Jorgenson} and \textsc{Lang} \cite{JL:SL2R,JL:SL2C}
which contain more specialized expositions of the subject.

The \emph{spherical transform} and the \emph{Harish transform} can be defined 
either for bi-invariant functions on $G=\GL(n,\mathbb{C})$ \cite{HarCha1,HarCha2,Helgason3,JL:SL2R,JL:SL2C}
or for invariant functions on $\Omega=\Pos(n,\mycmplx)$ \cite{FK,Terras}.
Indeed, in view of the bijection \eqref{eq:bijection}, 
these approaches are essentially equivalent.
We find it more convenient to define the transforms for functions on~$\Omega$.

To introduce the \emph{Harish transform}\footnote{The word \emph{Harish} indeed stands for the mathematician Harish-Chandra. To avoid confusion with the spherical transform, which is sometimes called the \emph{Harish-Chandra transform}, we follow Jorgenson and Lang \cite{JL:SL2R} and use the shorter name \emph{Harish transform}, although it unsatisfactorily misses the second part of the name.}, let us first recall the Cholesky decomposition from Remark \ref{rem:factorizations} (i).
Then the \emph{Harish transform} $\harish: L^{1,K}(\Omega)\rightarrow L^{\mathcal{H}}(A)$ of a density $f_\Omega\in L^{1,K}(\Omega)$ on $\Omega$  is defined by
\begin{align}
\label{eq:harish-transform}
\harish f_\Omega(a) := \left( \prod_{j=1}^{n} a_j^{(n-2j+1)/2} \right) \int_T f_\Omega(t^* a t) \, dt .
\end{align}
Let us note that, by Remark \ref{rem:factorizations}\,(i), we have
\begin{equation}
\int_A \left( \prod_{j=1}^{n} a_j^{2(n-j)} \right) \int_T |f_\Omega(t^* a t)| \, dt \, da = \int_\Omega |f(y)| \, dy < \infty \,,
\end{equation}
which implies that the integral in \eqref{eq:harish-transform} is well-defined
for almost all $a \in A$. That the resulting function $\harish f_\Omega$
lies indeed in $L^{\mathcal H}(A)$ follows from the following lemma,
which we also employ when relating the singular value and eigenvalue statistics.

\begin{lemma}[Factorization of $\mathcal{H}$]\label{lem:Harish}\

 The Harish-transform factorizes into the operators
 \begin{equation}\label{Harish-fact}
 \harish=\mathcal{I}_Z^{-1}\mathcal{T}\mathcal{I}_\Omega^{-1}:\ L^{1,K}(\Omega)\rightarrow L^{\mathcal{H}}(A).
 \end{equation}
 We recall the definitions~\eqref{I-Om-def}, \eqref{T-def}, and \eqref{I-Z-def}.
\end{lemma}

\begin{proof}
Let $f_\Omega \in L^{1,K}(\Omega)$. Then,
using Eqs.~\eqref{I-Om-def}, \eqref{T-def} and \eqref{I-Z-def}, we have
\begin{multline}
  \mathcal{T} \mathcal{I}_\Omega^{-1} f_\Omega(z) 
= \frac{1}{n! \, \pi^n}|\Delta_n(z)|^2 \left(\prod_{j=1}^n|z_j|^{2(n-j)}\right) 
		\int_{T} f_\Omega(t^* |z|^2 t) dt \\
= \frac{1}{n! \, \pi^n}|\Delta_n(z)|^2 \det(|z|)^{n-1} \, \mathcal{H} f_\Omega(|z|^2)
= \mathcal{I}_Z \mathcal{H} f_\Omega (z) \,,
\end{multline}
\ie $\mathcal{T} \mathcal{I}_\Omega^{-1} = \mathcal{I}_Z \mathcal{H}$.
Since $\mathcal{I}_Z: L^{\mathcal{H}}(A)\rightarrow L^{1,{\rm EV}}(Z)$ is a bijection,
this completes the proof.
\end{proof}

We emphasize that the Lemma~\ref{lem:Harish} in principle relates the joint densities of the squared singular values and of the eigenvalues. We have only to apply $\mathcal{I}_A^{-1}$ from the right and $\mathcal{I}_Z$ from the left yielding the operator $\mathcal{R}=\mathcal{T}\mathcal{I}_\Omega^{-1}\mathcal{I}_A^{-1}$. However, so far, it is neither clear whether this operator is bijective nor whether the integrals over $T$ can be simplified. This is the reason why we need the spherical transform introduced in Subsection~\ref{subsec:SphericalTransform}.

% ****************************************************************************************************

\subsection{Spherical Transform}\label{subsec:SphericalTransform}

The spherical transform is a multivariate analogue of the Mellin transform.
To introduce it, let us first define the \emph{generalized power function}
\begin{align}
\label{eq:power-function}
p(y,s+\varrho) :=(\det y)^{s_n+(n-1)/2}\prod_{j=1}^{n-1}(\det \Pi_jy\Pi_j^* )^{s_j-s_{j+1}-1}
\end{align}
of a matrix $y\in \Omega$ and a diagonal matrix $s = \diag(s_1,\hdots,s_n) \in \mycmplx^n$ and $\varrho$ as in Eq.~\eqref{eq:varrho}.
Here the $j \times n$ matrix $\Pi_j=(\eins_j,0_{j\times(n-j)})$ yields the projection onto the first $j$ rows, with $\eins_j$ the $j \times j$ identity matrix and $0_{j\times(n-j)}$ the $j\times(n-j)$ zero matrix. Thus, the determinants $\det \Pi_jy\Pi_j^*$ are simply the principal minors of the matrix $y$.
The definition with the shift $\varrho$ is merely a matter of convenience. For instance, 
it makes the spherical functions and the spherical transform symmetric in the parameter $s$; 
see below.
 
The generalized power function~\eqref{eq:power-function} plays the role of the power functions $\lambda\mapsto\lambda^c$ in the Mellin transform, cf. Eqs.~\eqref{M-def} and \eqref{mult-M-def}. However the generalized power function~\eqref{eq:power-function} is not invariant under the action of the unitary group on $y$ via the adjoint action, in particular it depends on the order of the rows and columns of $y$. The correct replacement for the power function in the Mellin transform for the matrix space $\Omega$ is the spherical function $\varphi(.,s):\Omega\rightarrow\mathbb{C}$,
which has the desired invariance property in the first argument.
In general, the spherical functions $\varphi$ may be characterized by the properties
that they are not identically zero and that they satisfy the non-linear integral equation
\begin{equation}\label{sphe-def}
\int_K \varphi(y_1^{1/2}k^*y_2ky_1^{1/2},s)d^*k=\varphi(y_1,s)\varphi(y_2,s), \text{ for all }y_1,y_2\in\Omega \,;
\end{equation}
compare \eg \cite[Proposition IV.2.2]{Helgason3}.
It turns out that these functions may be parametrized by the diagonal matrix $s\in\mathbb{C}^n$,
and defined by the group integral
\begin{align}
\label{eq:spherical-function}
\varphi(y,s) = \int_K p(k^*yk,s+\varrho) \, d^*k \,;
\end{align}
compare \eg \cite[Theorem IV.4.3]{Helgason3}.

It is obvious that the spherical function is $K$-invariant by construction. Therefore we may specify $\varphi(.,s)$ via its restriction to $A\subset\Omega$. Let $\lambda(y)\in A$ be the diagonal matrix of the eigenvalues of $y\in\Omega$. Then the spherical function admits the explicit representation 
due to Gelfand and Na\u{\i}mark~\cite{GelNai}
\begin{eqnarray}
\label{eq:gelfand-naimark}
  \varphi(y,s) 
&=& \frac{\Delta_n(\varrho)}{\Delta_n(s)} \, \frac{\det[(\lambda_j(y))^{s_{k}}]_{j,k=1,\hdots,n}}{\det[(\lambda_j(y)]^{\varrho_{k}})_{j,k=1,\hdots,n}} \\
&=& \left(\prod_{j=0}^{n-1}j!\right)\frac{\det\left[(\lambda_j(y))^{s_k+(n-1)/2}\right]_{j,k=1,\hdots,n}}{\Delta_n(s) \, \Delta_n(\lambda(y))} \,,\nonumber
\end{eqnarray}
see also \cite[Theorem IV.5.7]{Helgason3}, \cite[Theorem XII.4.3]{JL:SL2R}, as well as Remark~\ref{rem:Adaption} below. 
The representation \eqref{eq:gelfand-naimark} is valid only for pairwise different $s_j$ and $\lambda_j(y)$. 
When several of these arguments are equal, we have to apply L'H\^opital's rule. 
We~also notice that  the spherical functions are sort of continuous interpolations
of the \emph{Schur polynomials} in our setting.

Equation~\eqref{sphe-def} also fixes the normalization by choosing $y_2=\eins_n$, i.e. $\varphi(y_1,s)=\varphi(y_1,s)\varphi(\eins_n,s)$ for all $y_1\in\Omega$ and, hence, $\varphi(\eins_n,s)=1$, which is indeed satisfied by Eqs.~\eqref{eq:spherical-function} and \eqref{eq:gelfand-naimark}. Furthermore, we have a symmetry under the change $\varrho \to -\varrho$ which is equal to a permutation of the entries of $\varrho$. Therefore $\varphi(y_1,\varrho)=\varphi(y_1,-\varrho)=1$. 
This follows from the general fact that $\varphi(y,s)$ is invariant under the Weyl group of $K$ acting on $s$, 
see \eg Ref.~\cite[Theorem XIV.3.1]{FK}, which is in~our case the symmetric group $\mathbb{S}_n$.

With the help of the spherical function we can define the \emph{spherical transform} $\mys: L^{1,K}(\Omega)\rightarrow\mathcal{M}L^{1,\mathcal{H}}(A)$ for a $K$-invariant density $f_\Omega\in L^{1,K}(\Omega)$ by
\begin{align}
\mys f_\Omega(s) := \int_\Omega f_\Omega(y) \varphi(y,s)  \frac{dy}{(\det y)^{n}}
            = \int_\Omega f_\Omega(y') p(y',s+\varrho) \frac{dy'}{(\det y')^{n}},
\label{eq:spherical-transform}
\end{align}
for all those $s \in \mathbb{C}^n$ for which the integrand is Lebesgue integrable. 
To see that the~two integrals in Eq.~\eqref{eq:spherical-transform} are equal, 
we substitute $y'=k^* y k$ and use the $K$-invariance of $f_\Omega$.
Using that $\varphi(y,\varrho) \equiv 1$, it is easy to see that if $f_\Omega \in L^1(\Omega)$, 
then $\mys f_\Omega(s)$ exists for all $s \in \varrho' + \imath \myreal^n$,
where the shift $\varrho'$ is defined as in Eq.~\eqref{eq:varrho}.
More generally, it can be shown~\cite[Theorem IV.8.1]{Helgason3}
that if $f_\Omega \in L^{1}(\Omega)$, then $\mys f_\Omega (s)$ exists
at least in the tube $\operatorname{conv}(\mathbb{S}_n(\varrho')) + \i\myreal^n$,
where $\operatorname{conv}(\mathbb{S}_n(\varrho'))$ denotes the convex hull
of the orbit of $\varrho'$ under the Weyl group $\mathbb{S}_n$.

\begin{remark}[Adaption of Notation]
\label{rem:Adaption}\

We emphasize that our definitions and assumptions are slightly different from those in part 
of the literature, which explains the differences in the cited results. For instance,
our spherical functions are defined on $\Omega=\Pos(n,\mycmplx)$ instead of $G=\GL(n,\mycmplx)$,
which amounts to adding a factor $1/2$ in the parameter $s$ of the spherical transform.
Also, note that we define integrability with respect to the Lebesgue measure $dy$ on $\Omega$,
whereas the literature usually defines integrability with respect to the $G$-adjoint-invariant measure
$d^* y = (\det y)^{-n} \, dy$ on~$\Omega$. This explains the additional offset $n$ in the tube
 $\varrho' + \i\myreal^n$, cf. Lemma~\ref{lem:sets}.
\end{remark}

To see that $\mys L^{1,K}(\Omega)$ is indeed contained in $\mellin L^{\mathcal{H}}(A)$,
we recall a factorization theorem proven in Refs.~\cite[Lemma 43]{HarCha2} and~\cite[Proposition III.5.1]{JL:SL2R} 
which we state here as a lemma.
 
\begin{lemma}[Factorization of $\mathcal{S}$]\label{lem:spheric}\ 

 The  spherical transform factorizes into the multivariate Mellin transform and the Harish transform as
 \begin{equation}\label{spheric-fact}
 \mys = \mellin \harish:\ L^{1,K}(\Omega)\rightarrow \mathcal{M}L^{1,\mathcal{H}}(A).
 \end{equation}
 We recall the definitions~\eqref{mult-M-def} and \eqref{eq:harish-transform}.
\end{lemma}

\begin{proof}
 We start from the second integral representation of Eq.~\eqref{eq:spherical-transform} and apply the Cholesky decomposition $y'=t^*at$ with $a\in A$ and $t\in T$. This yields
\begin{eqnarray}\label{spher-calc-a}
\mys f_\Omega(s) = \int_A\left(\int_T f_\Omega(t^*at) p(t^*at,s+\varrho) dt\right)\left(\prod_{j=1}^n a_j^{n-2j}\right)da.
\end{eqnarray}
The principal minors are equal to
\begin{equation}\label{spher-calc-b}
\det \Pi_jy\Pi_j^*=\det \Pi_jt^*at\Pi_j^*=\prod_{l=1}^j a_l
\end{equation}
because of $t\Pi_j^*=\Pi_j^* t_j$ with $t_j$ the $j\times j$ upper left block of $t$, which is also a uni\-triangular matrix. Let us point out that the determinant of any upper unitrangular matrix is equal to $1$. Thus we have 
\begin{eqnarray}\label{spher-calc-c}
\mys f_\Omega(s) = \int_A\left(\int_T f_\Omega(t^*at) dt\right) \left(\prod_{j=1}^n a_j^{s_j+(n-2j-1)/2}\right)da.
\end{eqnarray}
Comparison with the definition of the Harish transform~\eqref{eq:harish-transform} allows us to simplify this intermediate result to
\begin{eqnarray}\label{spher-calc-d}
\mys f_\Omega(s) = \int_A\mathcal{H}f_\Omega(a) \left(\prod_{j=1}^n a_j^{s_j-1}\right)da.
\end{eqnarray}
Recall from Lemma \ref{lem:Harish} that 
$\mathcal{H} L^{1,K}(\Omega) \subset L^{\mathcal{H}}(A)$,
the set on which we have defined  the multivariate Mellin transform \eqref{mult-M-def}.
Furthermore, again by Lemma \ref{lem:Harish}, the Harish transform is symmetric in its argument $a$,
so that we can symmetrize the product in Eq.~\eqref{spher-calc-d} to obtain
\begin{eqnarray}\label{spher-calc-e}
\mys f_\Omega(s) = \frac{1}{n!}\int_A\mathcal{H}f_\Omega(a) {\rm Perm}[a_b^{s_c-1}]_{b,c=1,\ldots,n}da=\mathcal{M}\mathcal{H}f_\Omega(a) \,,
\end{eqnarray}
which concludes the proof.
\end{proof}

Let us state and prove a variant of the spherical inversion formula~\cite[Chapters IV.3 and IV.8]{Helgason3} which 
is the analogue of the inversion formula of the Mellin inversion formula from Subsection~\ref{subsec:MellinTransform}. 
In particular, this inversion formula shows that $\mathcal{S}$ is~injective,
i.e. if $f_1,f_2\in L^{1,K}(\Omega)$ and $\mys f_1$ and $\mys f_2$ on $\varrho' + \imath \myreal^n$,
then $f_1=f_2$ almost everywhere.

\begin{lemma}[Spherical Inversion on $L^{1,K}(\Omega)$] \label{lem:sphe-inv}\

\noindent{}The inverse of the spherical transform $\mys:L^{1,K}(\Omega)\rightarrow \mathcal{S}L^{1,K}(\Omega) = \mathcal{M} L^{\mathcal{H}}(A)$ is given by
% where $f_\Omega\in L^{1,K}(\Omega)$ is continuous by
\begin{multline}
 \mathcal{S}^{-1}([\mys f_\Omega],y)
 :=\frac{1}{n! \pi^{n(n-1)/2}}\frac{1}{\Delta_n(\lambda(y))}\lim_{\epsilon\to0}\int_{\mathbb{R}^n}
 \zeta_n(\epsilon s) % \zeta_n(\epsilon (s - \imath \sigma)) 
 \mathcal{S} f_\Omega(\varrho'+\imath s) \\
 \times\Delta_n(\varrho'+\imath s)\det[(\lambda_b(y))^{-c-\imath s_c}]_{b,c=1,\ldots,n}\prod_{j=1}^n\frac{ds_j}{2\pi}
 = f_\Omega(y) \,, \label{spher-inv}
\end{multline}
for almost all $y\in\Omega$,
where $\zeta_n$ is defined as in Eq.~\eqref{zet-def}
% $\sigma :=\varrho' -(n-1)/2\eins_n=(1,\hdots,n)$,
and $\lambda(y)\in A$ is the diagonal matrix 
of the eigenvalues of $y\in\Omega$ as usual.
\end{lemma}

\begin{proof}
The identity $\mathcal{S}L^{1,K}(\Omega) = \mathcal{M} L^{\mathcal{H}}(A)$ is clear 
by Lemma \ref{lem:spheric} and the fact that $\mathcal{H} L^{1,K}(\Omega) = L^{\mathcal{H}}(A)$,
as follows from the factorization in Lemma \ref{lem:Harish} and the definition of the space 
$L^{\mathcal{H}}(A)$.

For the proof of \eqref{spher-inv}, fix $y \in \Omega$ such that 
\begin{align}
\label{eq:choice-of-y}
|\Delta_n(\lambda(y))| \ne 0 \quad\text{and}\quad \lim_{\varepsilon \to 0} \int_{B(0,1)} |f_\Omega(\lambda(y)+\varepsilon h) - f_\Omega(\lambda(y))| \, dh = 0 \,,
\end{align}
where the integral is over the $n$-dimensional unit ball $B(0,1)$.
Let us note that almost all $y \in \Omega$ satisfy these conditions.
The former condition follows from the fact that the eigenvalues of $y$ 
are pairwise different for almost all $y \in \Omega$,
and the latter condition may be deduced from Eq.~\eqref{eq:DifferentiationTheorem}
with $f = \FSV := \mathcal{I}_A f_\Omega$ and Eq.~\eqref{I-A-def}.

We first replace $\zeta_n(\epsilon s)$ with $\zeta_n(\epsilon (s-\imath \sigma))$,
where $\sigma :=\varrho' - (1/2)(n-1)\eins_n=(1,\hdots,n)$.
To justify this step, note that we have 
\begin{multline}
\label{eq:spherical1}
\lim_{\epsilon\to0}\int_{\mathbb{R}^n}
 \big| \zeta_n(\epsilon (s - \imath \sigma)) - \zeta_n(\epsilon s) \big|
 \big| \mathcal{S} f_\Omega(\varrho'+\imath s) \big| \\
 \times \big| \Delta_n(\varrho'+\imath s) \big| \big| \det[(\lambda_b(y))^{-c-\imath s_c}]_{b,c=1,\ldots,n} \big| \prod_{j=1}^n\frac{ds_j}{2\pi} = 0
\end{multline}
by the dominated convergence theorem, because $\zeta_n(\epsilon s)$ is continuous
and bounded by ${\rm const} \, (1+ \| s \|^2)^{-n}$ in any tube of bounded width around $\myreal^n$,
$\Delta_n(\varrho'+\imath s)$ is bounded by ${\rm const} \, \|s\|^{n-1}$, 
and all the other terms under the integral are bounded.

Now consider the definition of the spherical transform \eqref{eq:spherical-transform}.
Using the spectral decomposition $y'=k^*a k$ with $k\in K$ and $a\in A$, 
see Eq.~\eqref{spec-dec} for the change of the measure, 
as well as the $K$-invariance of $f_\Omega$ and $\varphi(\,\cdot\,,s)$,
we have 
\begin{align}
\label{inv-clac-n}
\mathcal{S} f_\Omega(\varrho'+\imath s) = \left(\frac{1}{n!}\prod_{j=0}^{n-1}\frac{\pi^{j}}{j!}\right) 
	\int_A f_\Omega(a) \varphi(a,\varrho'+\imath s) |\Delta_n(a)|^2 \frac{da}{(\det a)^n} \,.
\end{align}
where the integral over $A$ exists as a Lebesgue integral. 
In particular, this also holds for $s = 0$.

Inserting Eq.~\eqref{inv-clac-n} into Eq.~\eqref{spher-inv} 
(but with $\zeta_n(\epsilon s)$ replaced by $\zeta_n(\epsilon (s - \imath \sigma))$)
and using the simple estimate $|\varphi(a,\varrho'+\imath s)|\leq \varphi(a,\varrho')$,
we find that the resulting integrand (viewed as a function of $a$ and $s$)
is bounded, up to a constant, by
\begin{multline}
\label{inv-clac-o}
\bigg|\zeta_n(\epsilon (s - \imath \sigma)) \Big(f_\Omega(a) \varphi(a,\varrho'+\imath s)\frac{|\Delta_n(a)|^2}{(\det a)^n}\Big) 
	\Delta_n(\varrho'+\imath s) {\rm det}[(\lambda_b(y))^{-c}]_{b,c=1,\ldots,n}\bigg| \\
\leq \big|\zeta_n(\epsilon (s - \imath \sigma))\Delta_n(\varrho'+\imath s)\big| 
	\Big(|f_\Omega(a)| \varphi(a,\varrho')\frac{|\Delta_n(a)|^2}{(\det a)^n}\Big)
	{\rm Perm}[(\lambda_b(y))^{-c}]_{b,c=1,\ldots,n} \,.
\end{multline}
Here the first factor is integrable with respect to $s$ 
by the estimates below Eq.~\eqref{eq:spherical1},
the second factor is integrable with respect to $a$ by the previous argument,
and the third factor is bounded (as $y$ is fixed).
Thus, the integrand in Eq.~\eqref{spher-inv} is Lebesgue integrable in $a$ and $s$,
and we may interchange the integrations over $s$ and $a$.

Inserting the representation \eqref{eq:gelfand-naimark} for the spherical function and simplifying, 
we~therefore obtain
\begin{multline}
\label{inv-clac-p}
\mys^{-1}([\mys f_\Omega],y) = 
\frac{1}{(n!)^2 \Delta_n(\lambda(y))}\lim_{\epsilon\to0}\int_A f_\Omega(a) \bigg( \int_{\mathbb{R}^n}\zeta_n(\epsilon (s-\imath \sigma)) \\ \times
	\det\left[a_j^{k+\imath s_k}\right]_{j,k=1,\hdots,n} 
	\det\left[(\lambda_b(y))^{-c-\imath s_c}\right]_{b,c=1,\ldots,n} 
	\prod_{j=1}^n \frac{ds_j}{2\pi} \bigg) \Delta_n(a) \, \frac{da}{\det a} \,.
\end{multline}
In the inner integral, we may now make the substitions $s_j \to s_j + \imath j$ 
and deform the resulting contours back to the real axis.
After that, the integrand is symmetric in~$s$,
and we apply the Andr\'eief identity \cite{Andreief} to obtain
\begin{eqnarray}
&&\int_{\mathbb{R}^n}\zeta_n(\epsilon s) \det\left[a_j^{\imath s_k}\right]_{j,k=1,\hdots,n} 
	\det\left[(\lambda_b(y))^{-\imath s_c}\right]_{b,c=1,\ldots,n} 
	\prod_{j=1}^n \frac{ds_j}{2\pi}\nonumber \\
	&=& n! \det \left[ \int_{\myreal} \widetilde\zeta_n(\varepsilon t) \left( \frac{a_j}{\lambda_k(y)} \right)^{\imath t} \frac{dt}{2\pi} \right]_{j,k=1,\hdots,n}\nonumber\\
&=& n! \det \left[ \frac{1}{\varepsilon} \fourier \widetilde\zeta_n \left( \frac{1}{\varepsilon} {\rm ln}\biggl(\frac{a_j}{\lambda_k(y)}\biggl) \right) \right]_{j,k=1,\hdots,n}
\label{inv-clac-q}
\end{eqnarray}
where
\begin{equation}\label{tildezeta}
\widetilde\zeta_n(t):=\frac{\pi^{2n}\cos t}{\prod_{k=1}^{n}(\pi^2-4t^2/(2k-1)^2)},
\end{equation}
with the Fourier transform
\begin{equation}\label{F-tildezeta}
 \mathcal{F}\widetilde\zeta_n\left(u\right):=c \, 
\Theta\left[1-u^2\right] \cos^{2n-1} \left[\frac{\pi u}{2}\right]
 \end{equation}
and $c$ its normalization constant which is not that important.

There are only two things we need to know here.
First, the integration over $a_j$ restricts to a compact domain 
$\bigcup_{k=1}^{n} [\lambda_k(y)e^{-\epsilon},\lambda_k(y) e^{\epsilon}]$
due to the Heaviside step function, 
and, second, the Fourier transform is normalized,
i.e. $\int_{\myreal} \fourier \widetilde\zeta_n(u) \, du = 1$,
because $\widetilde\zeta_n(0)=1$. 
 
Now recall that we started from a value $y \in \Omega$ such that $|\Delta_n(y)| \ne 0$.
Thus, for sufficiently small $\varepsilon > 0$, 
$|\Delta_n(a)|$ is bounded away from zero and infinity
on the domain of integration (say $D \subset \myreal^n$),
and since we know that $\FSV$ is integrable over $D$, we~may infer 
from Eq.~\eqref{I-A-def} that $f_\Omega$ is integrable over $D$ as well.
Thus, since all the other functions are bounded on $D$,
we may expand the Vandermonde determinant in \eqref{inv-clac-p}
and the determinant in Eq.~\eqref{inv-clac-q}, and we obtain (exploiting symmetry)
\begin{multline}
\label{inv-clac-s}
\mys^{-1}([\mys f_\Omega],y) = 
\frac{1}{\Delta_n(\lambda(y))} \lim_{\epsilon\to0} \int_A f_\Omega(a)
\left( \prod_{j=1}^n \frac{1}{\epsilon} \mathcal{F}\widetilde\zeta_n\left(\frac{1}{\epsilon}
	{\rm ln}\biggl(\frac{a_j}{\lambda_j(y)}\biggl)\right) \right)
 \, \Delta_n(a) \frac{da}{\det a} \,.
\end{multline}
Substituting $a_j = \lambda_j(y) e^{\varepsilon a_j'}$
and setting $\lambda(y) e^{\varepsilon a'} := (\lambda_1(y) e^{\varepsilon a_1'},\hdots,\lambda_n(y) e^{\varepsilon a_n'})$ for~abbreviation, we further obtain
\begin{align}
\label{inv-clac-t}
\mys^{-1}([\mys f_\Omega],y) = 
\lim_{\epsilon\to0} \int_{[-1,+1]^n} f_\Omega(\lambda(y) e^{\varepsilon a'})
\left( \prod_{j=1}^n \mathcal{F}\widetilde\zeta_n\left(a_j'\right) \right) \frac{\Delta_n(\lambda(y) e^{\varepsilon a'})}{\Delta_n(\lambda(y))} \, da' \,.
\end{align}
Now a similar argument as in the proof of Lemma \ref{lem:Mel-Inv} shows that 
the limit is equal to $f_\Omega(\lambda(y))$, and hence to $f_\Omega(y)$
by $K$-invariance of $f_\Omega$. This concludes the proof.
\end{proof}

On the one hand we are quite confident that the form of the inversion formula \eqref{spher-inv} can be extended to a certain class of distributions as for the Mellin inversion. 
% One has only to employ another regularizing test function than $\zeta_l$ in our case. 
On the other hand in applications like the Gaussian, Jacobi and even the Meijer G-ensembles, see Section~\ref{sec:MainResults}, 
we encounter densities where we do not need a regularization at all. In particular, $\zeta_n(\epsilon s)$ can be omitted 
when the rest of the integrand is absolutely integrable, possibly after an appropriate deformation of the contours.

% ****************************************************************************************************

\section{Main Results}\label{sec:MainResults}

Now we are ready to formulate our main results. We only have to put the pieces together that we have proven in the Lemmas~\ref{lem:Mel-Inv}, \ref{lem:Harish}, \ref{lem:spheric}, and \ref{lem:sphe-inv}. Thus, we obtain the following commutative diagram:
\begin{equation}\label{diagram}
\begin{array}{ccccc} \vspace*{-0.3cm} 
 L^{1,K}(G) & \overset{\displaystyle\mathcal{I}_\Omega}{ \text{\scalebox{3}[1]{$\rightarrow$}}} &  L^{1,K}(\Omega) & \overset{\displaystyle\mathcal{I}_A}{\text{\scalebox{3}[1]{$\rightarrow$}}} & L^{1,{\rm SV}}(A) \\   \vspace*{0.0cm}  \\
 \text{\scalebox{1}[3]{$\downarrow$}}\begin{minipage}{0.3cm} \vspace*{-0.5cm}$\mathcal{T}$\end{minipage} & & \text{\scalebox{1}[3]{$\downarrow$}}\begin{minipage}{0.3cm} \vspace*{-0.5cm}$\mathcal{H}$\end{minipage} & \hspace*{-0.3cm}\begin{minipage}{1cm}\vspace*{-0.6cm}\text{\scalebox{2.7}[1.6]{$\searrow$}}\end{minipage}\hspace*{-0.5cm}\begin{minipage}{0.3cm}\vspace*{-1cm} $\mathcal{S}$\end{minipage}  &   \\ \vspace*{-0.5cm} \\
 L^{1,{\rm EV}}(Z) & \overset{\displaystyle\mathcal{I}_Z}{\text{\scalebox{3}[1]{$\leftarrow$}}} & L^{\mathcal{H}}(A) & \overset{\displaystyle\mathcal{M}}{\text{\scalebox{3}[1]{$\rightarrow$}}}  & \mathcal{M}L^{\mathcal{H}}(A)
\end{array}
\end{equation}
It is commutative due to the factorizations of the Harish transform $\mathcal{H}$, see Lemma~\ref{lem:Harish}, and the spherical transform $\mathcal{S}$, see Lemma~\ref{lem:spheric}. Moreover we also know that all operators in this diagram are bijective on these spaces of densities since $\mathcal{I}_\Omega$,  $\mathcal{I}_A$, and  $\mathcal{I}_Z$ as well as $\mathcal{S}$ and $\mathcal{M}$ are bijective.  In subsection~\ref{subsec:Map} we compose the operators such that we find the corresponding map between the joint densities of the singular values $L^{1,{\rm SV}}(A)$ and those of the eigenvalues $L^{1,{\rm EV}}(Z)$ for general bi-unitarily invariant densities $L^{1,K}(G)$.

The diagram is indeed richer in its interpretation than only the relation between the eigenvalues and singular values, 
because the elements of $L^{1,K}(\Omega)$ are the $K$-invariant densities on the space of positive definite Hermitian matrices and the elements of $L^{\mathcal{H}}(A)$ are
closely related to the joint densities of the radii of the eigenvalues 
of bi-unitarily invariant random matrices; see Remark \ref{rem:RadiusDistribution} above. 
Only the image $\mathcal{M}L^{\mathcal{H}}(A)$ of the Mellin transform of $L^{\mathcal{H}}(A)$ is auxiliary.

In Subsections~\ref{subsec:PolynomialEnsembles} and \ref{subsec:break} we consider two direct applications of this new map. These applications are to polynomial ensembles and to a particular class of non-bi-unitarily invariant ensembles, respectively.

\subsection{Mapping between Singular Value and Eigenvalue Statistics}\label{subsec:Map}

Our first result is the bijective map between the space $L^{1,{\rm SV}}(A)$ of the densities for the squared singular values 
and the space $L^{1,{\rm EV}}(Z)$ of the densities for the eigenvalues of bi-unitarily  invariant densities on $G$.

\begin{theorem}[Map between $\FSV$ and $\FEV$]\label{thm:Map}\

The map
\begin{equation}\label{R-def}
 \mathcal{R}=\mathcal{T}\mathcal{I}_\Omega^{-1}\mathcal{I}_A^{-1}:\ L^{1,{\rm SV}}(A)\rightarrow L^{1,{\rm EV}}(Z)
\end{equation}
 from the joint densities of the singular values $L^{1,{\rm SV}}(A)=\mathcal{I}_A\mathcal{I}_\Omega L^{1,K}(G)$ to the joint densities of the eigenvalues $ L^{1,{\rm EV}}(Z)=\mathcal{T}L^{1,K}(G)$ induced by the bi-unitarily invariant signed densities $L^{1,K}(G)$ is bijective and has  the explicit integral representation
\begin{eqnarray}
\FEV(z)\hspace*{-0.1cm}&=&\hspace*{-0.1cm}\mathcal{R}\FSV(z)\label{R-rep}\\
&=&\hspace*{-0.1cm}\frac{\prod_{j=0}^{n-1}j!}{(n!)^2\pi^n}|\Delta_n(z)|^2\lim_{\epsilon\to0}\int_{\mathbb{R}^n}\zeta_1(\epsilon s){\rm Perm}\bigl[|z_b|^{-2(c+\imath s_c)}\bigl]_{b,c=1,\ldots,n}\nonumber\\
 &&\times \Bigl(\int_A \FSV(a)\frac{\det[a_b^{c+\imath s_c}]_{b,c=1,\ldots,n}}{\Delta_n(\varrho'+\imath s)\Delta_n(a)}\prod_{j=1}^n\frac{da_j}{a_j}\Bigl)\prod_{j=1}^n\frac{ds_j}{2\pi}\nonumber
\end{eqnarray}
with $\FSV\in L^{1,{\rm SV}}(A)$, $\varrho'=\diag(\varrho'_1,\ldots,\varrho'_n)$, $\varrho'_j=(2j+n-1)/2$ for $j=1,\ldots,n$, and $s$ embedded as a diagonal matrix,
and
\begin{eqnarray}\label{R-inv-rep}
\FSV(a)\hspace*{-0.1cm}&=&\hspace*{-0.1cm}\mathcal{R}^{-1}\FEV(a)\\
&=&\hspace*{-0.1cm}\frac{\pi^n}{(n!)^2\prod_{j=0}^{n-1}j!}\Delta_n(a)\lim_{\epsilon\to0}\int_{\mathbb{R}^n}\zeta_n\left(\epsilon s\right)\Delta_n(\varrho'+\imath s)\det[a_b^{-c-\imath s_c}]_{b,c=1,\ldots,n}\nonumber\\
&&\times\Bigl(\int_A {\rm Perm}[{a'}_b^{c+\imath s_c}]_{b,c=1,\ldots,n}
\, \frac{\FEV(\sqrt{a'})}{|\Delta_n(\sqrt{a'})|^2} \, \prod_{j=1}^n\frac{da'_j}{a'_j}\Bigl)\prod_{j=1}^n\frac{ds_j}{2\pi}\nonumber
\end{eqnarray}
with $\FEV\in L^{1,{\rm EV}}(Z)$ for its inverse
and the regularizing function
\begin{equation}\label{zet-def-b}
\zeta_l(s)=\prod_{j=1}^n \frac{\pi^{2l}\cos s_j}{\prod_{k=1}^l(\pi^2-4 s_j^2/(2k-1)^2)},\ l\in\mathbb{N}.
\end{equation}
We call $\mathcal{R}$ the \emph{SEV-transform}.
\end{theorem}

Note that the integral representation~\eqref{R-rep} is indeed a simplification compared to Eq.~\eqref{rel-ev-2} where we have to integrate over $T$. The number of integration variables is reduced from $n(n-1)$ for the integral over $T$ to $2n$ for the operator $\mathcal{R}$. Moreover we have an explicit representation of the inverse $\mathcal{R}^{-1}$ which was not known before, not to mention that it was known to be invertible.

As mentioned before, the regularizing functions $\zeta_n$ can usually be omitted in practice. 
Quite often we can deform the contour such that the integrand without $\zeta_n$ is Lebesgue integrable.

Let us emphasize that we could also have started from the set $L^{1,K}_{\text{prob}}(G)$ 
of all bi-unitarily invariant \emph{probability densities} on $G = \GL(n,\mathbb{C})$.
In this case, we would have obtained a similar bijection between the set $L^{1,{\rm SV}}_{\text{prob}}(A)$ 
of all symmetric probability densities on $A$
and the set $L^{1,{\rm EV}}_{\text{prob}}(Z) := \mathcal{T} L^{1,K}_{\text{prob}}(G)$ 
of all \emph{induced} symmetric probability densities on $Z$.
However, in general, this is only a subset of the set of all probability densities 
in $\mathcal{T} L^{1,K}(G)$; see Remark \ref{rem:correspondence} below.

\begin{proof}[Proof of Theorem \ref{thm:Map}]
Starting from the commutative diagram~\eqref{diagram} we have
 \begin{equation}\label{map-calc-a}
 \mathcal{R}=\mathcal{T}\mathcal{I}_\Omega^{-1}\mathcal{I}_A^{-1}=\mathcal{I}_Z\mathcal{H}\mathcal{I}_A^{-1}= \mathcal{I}_Z\mathcal{M}^{-1}\mathcal{S}\mathcal{I}_A^{-1}.
 \end{equation}
 Since all of the operators on the right hand side are invertible also $\mathcal{R}$ is invertible with $\mathcal{R}^{-1}=\mathcal{I}_A\mathcal{S}^{-1}\mathcal{M}\mathcal{I}_Z^{-1}$.
 
 The explicit representations of $\mathcal{R}$ and $\mathcal{R}^{-1}$ directly follow from those of the operators $\mathcal{I}_A$, $\mathcal{S}$, $\mathcal{M}$ and  $\mathcal{I}_Z$, see Eqs.~\eqref{I-A-def}, \eqref{eq:spherical-transform}, \eqref{eq:gelfand-naimark}, \eqref{spher-inv}, \eqref{mult-M-def}, \eqref{mult-MellinInversion} and \eqref{I-Z-def}, respectively. We first consider $\mathcal{R}$ then we have for the product $\mathcal{S}\mathcal{I}_A^{-1}$,
 \begin{equation}\label{map-calc-b}
  \mathcal{S}\mathcal{I}_A^{-1}\FSV(s)=\left(\prod_{j=0}^{n-1}j!\right)\int_A\FSV(a)\frac{\det[a_b^{s_c-(n-1)/2}]_{b,c=1,\ldots,n}}{\Delta_n(s)\Delta_n(a)}\frac{da}{\det a}
 \end{equation}
 for any $\FSV\in L^{1,{\rm SV}}(A)$. The product $\mathcal{I}_Z\mathcal{M}^{-1}$ explicitly reads
\begin{eqnarray}
\mathcal{I}_Z\mathcal{M}^{-1}\mathcal{S}\mathcal{I}_A^{-1}\FSV(z)&=& \frac{1}{(n!)^2\pi^n}|\Delta_n(z)|^2\lim_{\epsilon\to 0}\int_{\mathbb{R}^n} \zeta_1(\epsilon s)\mathcal{S}\mathcal{I}_A^{-1}\FSV(\varrho'+\imath s)\nonumber\\
&&\times {\rm Perm}[|z_b|^{-2(c+\imath s_c)}]_{b,c=1,\ldots,n}\prod_{j=1}^{n}\frac{ds_j}{2\pi}.
\label{map-calc-c}
\end{eqnarray}
The combination of Eqs.~\eqref{map-calc-b} and \eqref{map-calc-c} leads to the result~\eqref{R-rep}.

For $\mathcal{R}^{-1}$ we first consider the product $\mathcal{M}\mathcal{I}_Z^{-1}$ which is equal to
 \begin{equation}\label{map-calc-d}
  \mathcal{M}\mathcal{I}_Z^{-1}\FEV(s)=\pi^n\int_A {\rm Perm}[{a'}_b^{s_c-1}]_{b,c=1,\ldots,n} \, \frac{\FEV(\sqrt{a'})}{|\Delta_n(\sqrt{a'})|^2} \frac{da'}{(\det a')^{(n-1)/2}}
 \end{equation}
 for any $\FEV\in L^{1,{\rm EV}}(Z)$. The other product involved, $\mathcal{I}_A\mathcal{S}^{-1}$, is explicitly
\begin{eqnarray}
\mathcal{I}_A\mathcal{S}^{-1}\mathcal{M}\mathcal{I}_Z^{-1}\FEV(a)&=& \frac{\Delta_n(a)}{(n!)^2 \prod_{j=0}^{n-1}j!}\lim_{\epsilon\to0}\int_{\mathbb{R}^n}\zeta_n\left(\epsilon s\right)\nonumber\\
 &&\hspace*{-3cm}\times \mathcal{M}\mathcal{I}_Z^{-1}\FEV(\varrho'+\imath s)\Delta_n(\varrho'+\imath s)\det[a_b^{-c-\imath s_c}]_{b,c=1,\ldots,n}\prod_{j=1}^n\frac{ds_j}{2\pi}.
\label{map-calc-e}
\end{eqnarray}
Again we combine both intermediate results which yields the second equation of the theorem and completes the proof.
\end{proof}

It is quite remarkable that the SEV-transform $\mathcal{R}$ is bijective. Considering the singular value statistics it is clear that we can easily go back and forth between $f_G$ and $\FSV$ for bi-unitarily invariant ensembles since the integrals over the unitary group $K$ factorize and its normalized Haar measure is uniquely given. However the bijective relation between the densities $f_G$ and $\FEV$ is not that obvious. Neither the non-compact group integral in Eq.~\eqref{rel-ev} nor the integral over the unitriangular group $T$ in the Schur decomposition~\eqref{rel-ev-2} factorizes.
Contrary, it is worth emphasizing that apart from the Vandermonde factor,
the eigenvalue density depends only on the \emph{moduli} of the eigenvalues $z$.
Accepting this property,
then both, the density of the singular values and that of the eigenvalues, effectively depend 
only on $n$ free real positive parameters, the singular values and the radial parts of the eigenvalue.
In this respect, the existence of a bijection between the two densities comes perhaps not too unexpected. 
Regardless of whether this correspondence is puzzling or not, it is the bi-unitary invariance of the ensemble on $G$ which ensures this one-to-one correspondence. Hence we obtain the immediate consequences which we summarize in the following corollary. 
Note that this corollary is stated in terms of random matrices and, thus, for probability densities.
% while we discussed signed densities, as well, up to now. 
However, let point out that this corollary can also be extended to signed densities,
although the formulation is a bit cumbersome.
% We decided against it as a compromise of rigorousness and readability.

\begin{corollary}[$\FEV$ of $\FSV$ times Unitary Matrices]\label{cor:conclusions}\

Let $a\in A$ be a positive definite diagonal random matrix with the probability density $\FSV \in L^{1,{\rm SV}}(A)$.  
Then the joint density of the eigenvalues of each of the following random matrices on $G$
is given by $\FEV=\mathcal{R}\FSV$:

 \begin{itemize}
  \item[(a)] $k_1 ak_2$ with $k_1,k_2\in K$ distributed by the normalized Haar measure
  and $k_1,a,k_2$ independent,
  \item[(b)] $k_0 ak$ with a fixed $k_0\in K$ and $k\in K$ distributed by the normalized Haar measure, 
  and $a$ and $k$ independent,
  \item[(c)] $k ak_0$  with a fixed $k_0\in K$ and $k\in K$ distributed by the normalized Haar measure,
  and $a$ and $k$ independent.
 \end{itemize}
 
In particular, the choice $k_0=\eins_n$ is possible in (b) and (c).
\end{corollary}

\begin{proof}
The statement for the case (a) is clear by Theorem~\ref{thm:Map},
because $k_1 a k_2$ has the density $\mathcal{I}_\Omega^{-1}\mathcal{I}_A^{-1} \FSV$.
% The first case (a) for $\FEV=\mathcal{R}\FSV$ and the last statement about $\FSV=\mathcal{R}^{-1}\FEV$ 
% are obviously true because the application of the commutative diagram~\eqref{diagram} is certainly allowed. 
For the case (b) we note that $k_0 ak$ and $akk_0$ have the same eigenvalues. Due to the Haar measure we can absorb $k_0$. Thus the random matrices $akk_0$ and $ak$  share the joint density of the eigenvalues. The same is true for $k_1 a k_2$ and $ak$ with $k,k_1,k_2\in K$ independently distributed by the normalized Haar measure on $K$,
and independent from $a$. Therefore $k_0 ak$ of (b) and $k_1ak_2$ of (a) share the same joint density of their eigenvalues. Analogously one can prove case (c).
\end{proof}

Although the explicit results~\eqref{R-rep} and \eqref{R-inv-rep} look bulky and hard to handle, it is quite the opposite. In Subsection~\ref{subsec:PolynomialEnsembles} we consider the case of polynomial ensembles which have the nice property that they give rise to a determinantal point process. For these particular ensembles the SEV transform $\mathcal{R}$ drastically simplifies and we can even  formulate a statement when exactly the joint density has the simple form~\eqref{eq:EVD}.

In Subsection~\ref{subsec:break} we show how to generalize the relation between the joint densities of the eigenvalues and the singular values to certain classes of ensembles which even break the bi-unitary invariance via squeezing the complex spectrum in certain directions or creating repulsions from distinguished points in the spectrum. Hence the applicability of our results 
is far from being restricted to bi-unitarily invariant ensembles.

% ****************************************************************************************************

\subsection{First Application: Polynomial Ensembles}\label{subsec:PolynomialEnsembles}

\emph{Polynomial ensembles} recently introduced by \textsc{Kuijlaars}
and co-authors \cite{KS:2014,Kuijlaars:2015,CKW:2015,KKS:2015} are a specific kind of random matrix ensembles which have a certain structure for the joint density of their singular values. Hence they are defined via the density $\FSV\in L^{1,{\rm SV}}(A)$ instead of $f_G\in L^{1,K}(G)$. On the contrary, often the corresponding random matrix ensemble considered can be even not bi-unitarily invariant, e.g. see Refs.~\cite{Muirhead,ATLV:2004,RKGZ:2012,WWKS:2015} for the correlated Wishart and Jacobi ensemble or Ref.~\cite{FL:2015} for a product of  Ginibre ensembles times a single shifted Ginibre ensemble. Hence we have to apply the operators $\mathcal{K}_\Omega$ and $\mathcal{K}_A$, see Eqs.~\eqref{K-Om-def} and \eqref{K-A-def}, instead of  $\mathcal{I}_\Omega$ and $\mathcal{I}_A$. However once we have the joint density of the squared singular values we can apply Corollary~\ref{cor:conclusions} to think of another random matrix ensemble which is bi-unitarily invariant but sharing the same singular value statistics.

For this purpose we want to recall the definition of polynomial ensembles. Furthermore we introduce a subclass which has the nice property that the whole spectral statistics are governed by a single density on $\mathbb{R}_+$.

\begin{definition}[Polynomial Ensembles] \ \label{def:PE}

Fixing $n\in \mynat$ we choose $ n$ measurable functions $w_0,\hdots,w_{n-1}\in L^{1}_{[1,n]}(\mathbb{R}_+)$.
\begin{enumerate}[(a)]
\item 		The joint density $\FSV^{(n)}\in L^{1,{\rm SV}}(A)$ is called a \emph{polynomial ensemble} if and only if it satisfies the form~\cite{KS:2014,Kuijlaars:2015,CKW:2015}
\begin{align}
\label{eq:PE}
\FSV^{(n)}([w];a) := C_{\rm sv}^{(n)}[w]\Delta_n(a) \, \det[w_{j-1}(a_k)]_{j,k=1,\hdots,n}
\end{align}
with the normalization constant
\begin{equation}\label{norm:sv-gen}
\frac{1}{C_{\rm sv}^{(n)}[w]}:=n!\det\left[\int_0^\infty a^{k-1}w_{j-1}(a) da\right]_{j,k=1,\hdots,n}\in\mathbb{R}\setminus\{0\}.
\end{equation}
We underline that we allow in this definition signed densities, too.
\item		Assuming that there is a density $\omega\in L^{1,n-1}_{[1,n]}(\mathbb{R}_+)$ such that the two linear spans ${\rm span}\{w_0,\ldots,w_{n-1}\}$ and ${\rm span}\{\omega,(-a\partial_a)\omega,\ldots,(-a\partial_a)^{n-1}\omega\}$ agree then we call
\begin{align}
\label{eq:PE-der}
\FSV^{(n)}([\omega];a) := C_{\rm sv}^{(n)}[\omega]\Delta_n(a) \, \det[(-a_k\partial_{a_k})^{j-1}\omega(a_k)]_{j,k=1,\hdots,n}
\end{align}
with
\begin{equation}\label{norm:sv-der}
\frac{1}{C_{\rm sv}^{(n)}[\omega]}:=n!\det\left[\int_0^\infty a^{k-1} (-a\partial_{a})^{j-1}\omega(a) da\right]_{j,k=1,\hdots,n}\in\mathbb{R}\setminus\{0\}
\end{equation}
a \emph{polynomial ensemble of derivative type}.
\item		In the case that the function $\omega$ in (b) is a Meijer G-function, see Ref.~\cite{Abramowitz} for its definition, we call the ensemble a \emph{Meijer G-ensemble}.
\end{enumerate}
\end{definition}

Let us remark that the corresponding bi-unitarily invariant density $f_G^{(n)}[w]$ on $G$ can be easily obtained for a polynomial ensemble $\FSV^{(n)}[w]$,
\begin{equation}\label{FG-PE}
f_G^{(n)}([w];g)=C_{\rm sv}^{(n)}[w]\left(n!\prod_{j=0}^{n-1}\frac{(j!)^2}{\pi^{2j+1}}\right)\frac{\det\left[\tr\left((g^*g)^{b-1}w_{c-1}(g^*g)\right)\right]_{b,c=1,\ldots,n}}{\det[\tr(g^*g)^{b+c-2}]_{b,c=1,\ldots,n}}.
\end{equation}
For this result we multiplied the two determinants in the numerator as well as those in the denominator, e.g. $|\Delta_n(a)|^2=\det[\sum_{j=1}^n a_j^{b+c-2}]_{b,c=1,\ldots,n}$. For well-known ensembles like the Laguerre and the Jacobi ensemble one can easily show that it reduces to their standard representations. 
Note that the apparent singularities only appear for a degenerate spectrum of $g$ 
because the denominator corresponds to a squared Vandermonde determinant. However 
the zeros in the denominator cancel with those in the numerator,
since this determinant vanishes at these points, too.

\begin{examples}[Polynomial Ensembles of Derivative Types]\

We underline that quite a lot of the standard ensembles are Meijer G-ensembles. For example the induced Ginibre ensemble~\cite{Ginibre,Akemann-book,IK:2013} (also known as Laguerre, Wishart or chiral Gaussian unitary ensemble)
\begin{equation}\label{Gin-def}
 f_G^{(n)}([\omega_{\rm Lag}];g)\propto(\det g^*g)^\nu\exp[-\tr g^*g],\ \nu>-1,
\end{equation}
yields a Meijer G-ensemble with
\begin{equation}\label{om-Gin-def}
\omega_{\rm Lag}(a):=a^\nu e^{-a}
\end{equation}
because
\begin{eqnarray}\label{Delt-Gin}
\det[(-a_k\partial_{a_k})^{j-1}\omega_{\rm Lag}(a_k)]_{j,k=1,\ldots,n}=\Delta_n(a)(\det a)^\nu e^{-\tr a}.
\end{eqnarray}
The same is true for the induced Jacobi ensemble~\cite{ZS-trunc,Akemann-book,Forrester-book,IK:2013} 
(also known as the ensemble of truncated unitary matrices)
\begin{equation}\label{Jac-def}
 f_G^{(n)}([\omega_{\rm Jac}];g)\propto(\det g^*g)^\nu(\det(\eins_n-g^*g))^{\mu-n}\Theta(\eins_n-g^*g),\ \mu>n,\nu>-1,
\end{equation}
 yielding
\begin{equation}\label{om-Jac-def}
\omega_{\rm Jac}(a):=a^\nu (1-a)^{\mu-1}\Theta(1-a)
\end{equation}
and for the Cauchy-Lorentz ensemble~\cite{VKG-SUSY,Kieburg:2015,WWKS:2015}
\begin{equation}\label{CL-def}
 f_G^{(n)}([\omega_{\rm CL}];g)\propto\frac{(\det g^*g)^\nu}{(\det(\eins_n+g^*g))^{\nu+\mu+n}},\ \mu>n-1,\nu>-1,
\end{equation}
leading to
\begin{equation}\label{om-CL-def}
\omega_{\rm CL}(a):=\frac{a^\nu}{ (1+a)^{\nu+\mu+1}}.
\end{equation}
This can be seen by identities very similar to Eq.~\eqref{Delt-Gin}. The Heaviside step function $\Theta$ on matrix space is $1$ for positive definite matrices and vanishes otherwise. Note that the functions~\eqref{om-Gin-def}, \eqref{om-Jac-def} and \eqref{om-CL-def} are indeed Meijer G-functions~\cite{Abramowitz}.

Even the inverses of the ensembles above and their matrix products are in the class of Meijer G-ensembles, see Refs.~\cite{ARRS:2013,AB:2012,AIK:2013,AKW:2013,ABKN,Forrester:2014,IK:2013,KS:2014,KZ:2014,KKS:2015,AI:2015}. Therefore this class is already quite big,
and covers several important applications.

Indeed there are also other ensembles which are not Meijer G-ensembles but which are polynomial ensembles of derivative type. For example the choice
\begin{equation}\label{om-MB-def}
\omega_{\rm MB}(a):=a^\nu e^{-\alpha a^\theta},\ \alpha,\theta>0,\ \nu>-1,
\end{equation}
corresponds to the Laguerre-type of the Muttalib-Borodin ensemble~\cite{Muttalib, Borodin} with the joint probability density
\begin{equation}\label{MB-def}
\FSV^{(n)}([\omega_{\rm MB}];a)\propto\Delta_n(a)\Delta_n(a^\theta)(\det a)^\nu e^{-\alpha\tr a^\theta}.
\end{equation}
Only for integer $1/\theta$ the ensemble is a Meijer G-ensemble. One can even show that the Jacobi-type of the Muttalib-Borodin ensemble~\cite{Borodin,Forrester-Wang} falls into the class of polynomial ensembles of derivative type.

The limit of the Muttalib-Borodin ensemble for $\theta\to0$ can be modelled by~\cite{Forrester-Wang}
\begin{equation}\label{om-ln-def}
\omega_{\theta\to0}(a):=a^{\nu'} e^{-\alpha' ({\rm ln}\, a)^2},\ \alpha'=\theta^2\alpha/2>0,\ \nu'=\nu-\alpha\theta\in\mathbb{R},
\end{equation}
which leads to the joint probability density
\begin{equation}\label{ln-def}
\FSV^{(n)}([\omega_{\theta\to0}];a)\propto\Delta_n(a)\Delta_n({\rm ln}\, a)(\det a)^{\nu'} e^{-\alpha \tr({\rm ln}\, a)^2}.
\end{equation}
This ensemble is related to a particular ensemble studied in the theory of disorder conductors~\cite{Beenakker-Rejaei} 
which one obtains in the  limit of large eigenvalues~\cite{Forrester-Wang}.

For the opposite limit $\theta\to\infty$, we may use the transformation $a' = a^{\theta}$
to obtain the approximation
\begin{equation}\label{e-def}
\FSV^{(n)}([\omega_{ \theta\to\infty}];a')\propto\Delta_n(a')\Delta_n({\rm ln}\,a')(\det a')^{\nu''}e^{-\alpha \tr a'} \,,
\end{equation}
where $\nu'' := (\nu+1-\theta)/\theta>-1$.
It has the same kind of level repulsion for the singular values as the density in Eq.~\eqref{ln-def}.
However, the confining potential is different, namely $\tr a'$ instead of $\tr({\rm ln}\, a)^2$.
Moreover, the density in Eq.~\eqref{e-def} is not a polynomial ensemble of derivative type 
anymore, although it is still a polynomial ensemble.

The Muttalib-Borodin ensemble is also a good example that $\FSV$ may also correspond to a non-bi-unitarily invariant random matrix ensemble, see the recent work~\cite{Forrester-Wang} by \textsc{Forrester} and \textsc{Wang}. Certainly,
the relation between the eigenvalues and the singular values, see Theorem~\ref{thm:Map}, does not apply to 
these random matrix realizations. % otherwise it would be bi-unitarily invariant since $\mathcal{R}$ is injective.
\end{examples}

Let us consider now what the induced joint eigenvalue density will look like
for a density $\FSV^{(n)}\in L^{1,{\rm SV}}(A)$ which is a (possibly signed)
polynomial ensemble of a derivative type.
To this end, we have only to apply Theorem~\ref{thm:Map}.

\begin{theorem}[Joint Density of Eigenvalues]\label{thm:pol-SEV}\

Let $\FSV^{(n)}[\omega]\in L^{1,K}$ be the density of a polynomial ensemble of derivative type
with the function $\omega\in L^{1,n-1}_{[1,n]}(\mathbb{R}_+)$,
i.e. suppose that $\FSV$ is of the form~\eqref{eq:PE-der},
and let $f_G$ be the corresponding bi-unitarily invariant density in $L^{1,K}(G)$.
Then the corresponding joint density of the eigenvalues 
is equal to
	\begin{equation}\label{pol-der-SEV}
	\FEV^{(n)}([\omega];z)=\frac{C_{\rm sv}^{(n)}[\omega]\prod_{j=0}^{n-1}j!}{\pi^n}|\Delta_n(z)|^2\prod_{j=1}^n \omega(|z_j|^2).
	\end{equation}	
Additionally we can say for any bi-unitarily invariant density $f_G\in L^{1,K}(G)$  that the corresponding joint density of the squared singular values $\FSV\in L^{1,{\rm SV}}(A)$ is a polynomial ensemble of derivative type~\eqref{eq:PE-der} if and only if the corresponding joint density of the eigenvalues $\FEV\in L^{1,{\rm EV}}(Z)$ has the form~\eqref{pol-der-SEV}, with $\omega\in L^{1,n-1}_{[1,n]}(\mathbb{R}_+)$.
\end{theorem}

\begin{proof}
The operators and, thus, the integrals in the Theorem~\ref{thm:Map} are well defined for a polynomial ensemble of derivative type $\FSV^{(n)}[\omega]$, since it is Lebesgue integrable and symmetric by definition. The calculation is straightforward and we begin with performing the integral over $a$,
 \begin{eqnarray}\label{pol-calc-a}
 &&\int_A \FSV^{(n)}([\omega];a)\frac{\det[a_b^{c+\imath s_c}]_{b,c=1,\ldots,n}}{\Delta_n(\varrho'+\imath s)\Delta_n(a)}\frac{da}{\det a}\\
 &=&\frac{C_{\rm sv}^{(n)}[\omega]}{\Delta_n(\varrho'+\imath s)}\int_A \det[(-a_k\partial_{a_k})^{j-1}\omega(a_k)]_{j,k=1,\hdots,n} \, \det[a_b^{c+\imath s_c}]_{b,c=1,\ldots,n}\frac{da}{\det a}\nonumber\\
 &=&n!C_{\rm sv}^{(n)}[\omega]\frac{\det[(b+\imath s_b)^{c-1}\mathcal{M}\omega(b+\imath s_b)]_{b,c=1,\ldots,n}}{\Delta_n(\varrho'+\imath s)}\nonumber\\
 &=&n!C_{\rm sv}^{(n)}[\omega]\prod_{j=1}^n\mathcal{M}\omega(j+\imath s_j).\nonumber
 \end{eqnarray}
Here we have used the Andr\'eief formula~\cite{Andreief} and Eq.~\eqref{eq:mellin-link} and pulled the factors $\mathcal{M}\omega(b+\imath s_b)$ out of the rows of the determinant in the numerator which, then, becomes the Vandermonde determinant $\Delta_n(\varrho'+\imath s)$. The second integral in Eq.~\eqref{R-rep} in the variables $s$ is the multivariate Mellin inversion formula which yields the result~\eqref{pol-der-SEV}.

The remaining statement of the theorem is immediate since for bi-unitarily invariant ensembles the operator $\mathcal{R}$ is bijective and, thus, invertible.
\end{proof}

For random matrices from a polynomial ensembles of derivative type,
we can deduce the following consequence.

\begin{corollary}[Joint Density of Eigenvalues for Random Matrices]\label{cor:pol-sv-ev}\

Let $a\in A$ be a random matrix drawn from a polynomial ensemble of derivative type, 
with the function $\omega\in L^{1,n-1}_{[1,n]}(\mathbb{R}_+)$,
let $k_0$ be a fixed unitary matrix, and let $k$ be a random unitary matrix
with the normalized Haar measure on ${\rm U}(n)$ which is independent of $a$.
Then the induced eigenvalue density of each of the matrices
$k_0 a k$ and $k a k_0$ is given by Eq.~\eqref{pol-der-SEV}.
\end{corollary}

\begin{proof}
This follows from Theorem \ref{thm:pol-SEV} in the same way
as Corollary \ref{cor:conclusions} follows from Theorem \ref{thm:Map}.
\end{proof}

\begin{remark}[Correspondence between Singular Value and Eigenvalue Statistics]
\label{rem:correspondence}
Theorem \ref{thm:pol-SEV} unveils
a remarkably simple correspondence between
the (squared) singular value and eigenvalue densities 
for a wide class of \emph{bi-unitarily invariant} random matrix ensembles on $G=\GL(n,\mycmplx)$. 
For example, for the Muttalib-Borodin ensembles~\eqref{MB-def}, \eqref{ln-def} and \eqref{e-def},
this result is completely new. 
In general we solved in one strike a large number of integrals of the form
\begin{eqnarray}
 \mathcal{T}f_G^{(n)}([\omega];z)&\propto& |\Delta_n(z)|^2\left(\prod_{j=1}^{n} |z_j|^{2(n-j)}\right)\\
&&\times\int_{T}\frac{\det\left[\tr\left((t^*|z|^2t)^{b-1}(-a\partial_a)^{c-1}\omega|_{a=t^*|z|^2t})\right)\right]_{b,c=1,\ldots,n}}{\det[\tr(t^*|z|^2t)^{b+c-2}]_{b,c=1,\ldots,n}} dt.\nonumber
\end{eqnarray}
These integrals are by far non-trivial. For the particular case of $\omega(a)=a^\nu e^{-a^2}$ this becomes
\begin{eqnarray}
 \mathcal{T}f_G^{(n)}([a^\nu e^{-a^2}];z)&\propto& |\Delta_n(z)|^2\left(\prod_{j=1}^{n} |z_j|^{2(n-j+\nu)-1}\right)\\
&&\times\int_{T}\sqrt{\det(t^*|z|^2t\otimes\eins_n+\eins_n\otimes t^*|z|^2t)}e^{-\tr(t^*|z|^2t)^2} dt.\nonumber
\end{eqnarray}
This can be readily checked by the identity $\Delta_n(a^2)/\Delta_n(a)=\prod_{1\leq b<c\leq n}(a_b+a_c)=\sqrt{\det(a\otimes\eins_n+\eins_n\otimes a)}(\det 2a)^{-1/2}$. 

There is, however, one flaw regarding this correspondence between $\FSV^{(n)}[\omega]$ and $\FEV^{(n)}[\omega]$.
While a probability measure on the singular values always 
induces a probability measure on eigenvalues, a probability measure on the ``seeming'' eigenvalues of a random matrix ensemble does not need 
to correspond to a probability measure on the singular values.
Note that for a given joint probability density of the singular values,
there always exists an associated random matrix probability density, and hence a joint probability density for the eigenvalues.
The problem for the reverse direction is that the spherical inversion may result in
a \emph{signed} density for the singular values. As a simple example we consider a polynomial ensemble of derivative type with $\omega\geq0$ being a non-negative function. Then the corresponding $\FEV$, see Eq.~\eqref{pol-der-SEV}, is always a probability density (we will check the normalization in section~\ref{sec:Implications}). However the joint density of the squared singular values $\FSV$, see Eq.~\eqref{eq:PE-der}, is not necessarily positive. Indeed for the change of variables such that $\omega(a)=\exp[-V({\rm ln}\,a)]$ with $V$ the so called potential and the particular choice $n=2$, positivity of $\FSV^{(n=2)}$ implies the condition $(a_b-a_c)(V'({\rm ln}\,a_b)-V'({\rm ln}\,a_c))$ either non-negative or non-positive for all $a_b,a_c\in\mathbb{R}_+$. Hence $V$ has to be convex  everywhere on $\mathbb{R}$ due to the integrability of $\omega$.
\end{remark}

\begin{examples}[Generalizations to other polynomials]\

 We want to conclude this subsection by showing explicit relations between the joint densities $\FSV$ and $\FEV$ for other bi-unitarily invariant ensembles.
 \begin{enumerate}[(a)]
  \item		Considering a general polynomial ensemble of the form~\eqref{eq:PE} we cannot simplify all integrations as nicely as for the polynomials ensembles of derivative type. For instance, the joint density of the eigenvalues of $ak$ with $k\in K$ a Haar distributed unitary matrix and $a\in A$ an independent positive diagonal matrix drawn from the joint density $\FSV^{(n)}[w]$, a polynomial ensemble associated with the functions $w_0,\hdots,w_{n-1}\in L^{1}_{[1,n]}(\mathbb{R}_+)$, is given by
  \begin{eqnarray}
   \FEV^{(n)}([w];z)&=&\frac{C_{\rm sv}^{(n)}[w]\prod_{j=0}^{n-1}j!}{n!\pi^n}|\Delta_n(z)|^2\lim_{\epsilon\to0}\int_{\mathbb{R}^n}\zeta_1(\epsilon s){\rm Perm}\bigl[|z_b|^{-2c-2\imath s_c}\bigl]_{b,c=1,\ldots,n}\nonumber\\
 &&\times\frac{\det[\mathcal{M} w_{b-1}(c+\imath s_c)]_{b,c=1,\ldots,n}}{\Delta_n(\varrho'+\imath s)}\prod_{j=1}^n\frac{ds_j}{2\pi}.\label{pol-gen-SEV}
  \end{eqnarray}
  Alas, the remaining integral over $s$ cannot be easily performed in general.
  \item		Quite often we cannot reduce a polynomial ensemble to one of derivative type but we are ``very close" to it. For example the deformation
  \begin{equation}\label{ext-PE-ex}
  \FSV(a)\propto \det(a+\alpha_c\eins_n)^m e^{-\tr a} |\Delta_n(a)|^2
  \end{equation}
   with $\alpha\in\mathbb{C}$ a fixed complex variable and $m\in\mathbb{N}$ is no polynomial ensembles of derivative type. However we can generate this ensemble by a linear combination of densities
 with the following structure
  \begin{equation}\label{ext-PE-der}
  \FSV^{(n,m)}([\omega];a,B)\propto \Delta_n(a)\det[D[\omega],\ B],
  \end{equation}
  where $\omega\in L^{1,n+m-1}_{[1,n]}(\mathbb{R}_+)$,
  \begin{equation}\label{D-def}
   D[\omega]=\{(-a_c\partial_{a_c})^{b-1}\omega(a_c)\}_{\substack{b=1,\ldots,n+m \\ c=1,\ldots,n}}
  \end{equation}
  and
  \begin{equation}\label{B-def}
   B=\{\gamma_c^{b-1}\}_{\substack{b=1,\ldots,n+m \\ c=1,\ldots,m}}
  \end{equation}
   an $(n+m)\times m$ matrix of $m$ constant variables $\gamma_1,\ldots,\gamma_m\in\mathbb{C}$. We can expand in the parameters $\gamma_c$ to create the desired polynomial ensembles. However a direct relation between Eq.~\eqref{ext-PE-der} and the original density~\eqref{ext-PE-ex} is for an arbitrary $m\in\mathbb{N}$ non-trivial.
   
   A straightforward calculation involving a generalization of Andr\'eief's integra\-tion theorem, see Ref.~\cite[Appendix C.1.]{KG-det}, yields the joint density of the eigenvalues of the matrix $ak$, where $k\in K$ is again a Haar distributed unitary matrix and $a\in A$ is independent and distributed via the density~\eqref{ext-PE-der},
  \begin{eqnarray}
   \FEV^{(n,m)}([w];z)&\!\propto\!&\Delta_m(\gamma)|\Delta_n(z)|^2\prod_{j=1}^n \left(\left.\prod_{l=1}^m(\gamma_l+a\partial_a)\omega(a)\right|_{a=|z_j|^2}\right).\label{pol-ext-SEV}
  \end{eqnarray}
  Here we skipped the normalization.
  \item		We can also apply the second result~\eqref{R-inv-rep} of Theorem~\ref{thm:Map} and assume a joint density of the eigenvalues of a bi-unitarily invariant ensemble. A most natural generalization of the result for the polynomial ensembles of derivative type would be
  \begin{equation}\label{ev-ext}
  \FEV^{(n)}(z)\propto|\Delta_n(z)|^2{\rm Perm}[ \omega_b(|z_c|^2)]_{b,c=1,\ldots,n}
  \end{equation}
  $\omega_1,\ldots,\omega_n\in L^{1,n}_{[1,n]}(\mathbb{R}_+)$. Then the  singular value statistics of the corresponding bi-unitarily invariant ensemble reads
  \begin{equation}\label{ev-sv-ext}
  \FSV^{(n)}(z)\propto\Delta_n(a)\sum_{\sigma\in\mathbb{S}_n}\det[(-a_k\partial_{a_k})^{j-1}\omega_{\sigma(j)}(a_k)]_{j,k=1,\hdots,n}.
  \end{equation}
  The sum over the symmetric group encodes the former permanent. We skip the calculation since it is again straightforward.
\end{enumerate}
\end{examples}

\subsection{Breaking the bi-unitary invariance}\label{subsec:break}

Up to now we  mainly considered bi-unitarily invariant random matrix ensembles. However quite often one considers deformations of the ensemble breaking the bi-unitary invariance. A particular kind of deformations we want to consider is with the help of the following ensembles.

\begin{definition}[$G$-adjoint-invariant Deformations]\label{def:deform}\

 A $G$-adjoint-invariant deformation of a bi-unitarily invariant ensemble
 is a random matrix ensemble on $G$ whose density is given by
 \begin{equation}\label{def-deform}
  f_G:=f^{(K)}_G D_G\in L^1(G)
 \end{equation}
with $f_G^{(K)}\in L^{1,K}(G)$ and $D_G$ a function on $G$ which is $G$-adjoint-invariant, 
 i.e. $D_G(h^{-1}gh)=D_G(g)$ for all $h,g\in G$.
\end{definition}  

We underline that for a $G$-adjoint-invariant function we have $D_G(k^*gk)=D_G(g)$ for all $k\in K$ and $g\in G$ and
$D_G(zt)=D_G(z)$ for all $t\in T$ and $z\in Z$. The first statement is obvious because $k^*=k^{-1}$ and $K\subset G$. 
The second statement becomes clear after noticing that $zt$ has the eigenvalues $z$. Hence after an eigendecomposition $zt=h^{-1}zh$ with $h\in G$ the statement follows immediately. More generally, a function $f(g)$ on $G$ is 
$G$-adjoint-invariant if and only if it is a function of the eigenvalues of $g$.
Then we can prove the following theorem about the joint densities of the eigenvalues and singular values 
of this particular kind of deformation.

\begin{theorem}[Relation for $G$-adjoint-invariant Deformations]\label{thm:deform}\

 Let $f_G=f^{(K)}_G D_G\in L^1(G)$ be a $G$-adjoint-invariant deformation of a bi-unitarily invariant ensemble. 
 Then the joint density of the eigenvalues is equal to
 \begin{equation}\label{deform-ev}
  \FEV(z)=\mathcal{T}f_G(z)=D_G(z)\ \mathcal{T}f_G^{(K)}(z)
 \end{equation}
 and the joint density of the squared singular values is
 \begin{equation}\label{deform-sv}
  \FSV(a)=\mathcal{K}_A\mathcal{K}_\Omega f_G(a)=\left(\int_K D_G(\sqrt{a}k)d^*k\right) \mathcal{I}_A\mathcal{I}_\Omega f_G^{(K)}(a).
 \end{equation}
 In particular the bi-unitarily invariant parts are still related by
 \begin{equation}\label{Rel-deform}
  \mathcal{T}f_G^{(K)}(z)=\mathcal{R}\mathcal{I}_A\mathcal{I}_\Omega f_G^{(K)}(z),\ 
  \mathcal{I}_A\mathcal{I}_\Omega f_G^{(K)}(a)=\mathcal{R}^{-1}\mathcal{T}f_G^{(K)}(a).
 \end{equation}
\end{theorem}

Before we prove this theorem let us underline that the resulting joint density~\eqref{deform-sv} for the singular values can indeed be considered by itself as a density corresponding to a bi-unitarily invariant random matrix ensemble. Therefore the relation between the random matrix ensembles and those ensembles defined by their singular value density is by far unique. This does not contradict Corollary~\ref{cor:conclusions}. The relation between densities on $G$ and those on $A$ only becomes unique when assuming additional properties such as bi-unitary invariance.

\begin{proof}
The calculation is based on Eq.~\eqref{rel-comp} with the definitions~\eqref{T-def}, \eqref{K-Om-def} and \eqref{K-A-def}.  For the eigenvalues we have
\begin{eqnarray}\label{deform-calc-a}
\FEV(z)&=&\left(\frac{1}{n!}\prod_{j=0}^{n-1}\frac{\pi^{j}}{j!}\right)|\Delta_n(z)|^2\left(\prod_{j=1}^n|z_j|^{2(n-j)}\right)\\
&&\times\int_{T}\left(\int_{K} f_G^{(K)}(k^* z t k) D_G(k^*ztk) d^*k\right) dt\nonumber\\
&=&D_G(z)\left(\frac{1}{n!}\prod_{j=0}^{n-1}\frac{\pi^{j}}{j!}\right)|\Delta_n(z)|^2\left(\prod_{j=1}^n|z_j|^{2(n-j)}\right)\nonumber\\
&&\times\int_{T}\left(\int_{K} f_G^{(K)}(k^* z t k)  d^*k\right) dt.\nonumber
\end{eqnarray} 
The $t$ and $k$-dependence of $D_G$ drops out because of the $G$-adjoint-invariance, see the paragraph before 
Theorem~\ref{thm:deform}. The remaining integral is equal to $\mathcal{T}f_G^{(K)}$.

The joint density of the squared singular values is equal to
\begin{eqnarray}\label{deform-calc-b}
\FSV(a)&\mskip-11mu=\mskip-11mu&\left(\frac{1}{n!}\prod_{j=0}^{n-1}\frac{\pi^{2j+1}}{(j!)^2}\right)|\Delta_n(a)|^2\\
&&\qquad\times\int_K\left(\int_K f_G^{(K)}(k_1 \sqrt{k_2^*ak_2})D_G(k_1 \sqrt{k_2^*ak_2})d^*k_2\right) d^*k_1\nonumber\\
&\mskip-11mu=\mskip-11mu&\left(\frac{1}{n!}\prod_{j=0}^{n-1}\frac{\pi^{2j+1}}{(j!)^2}\right)|\Delta_n(a)|^2f_G^{(K)}(\sqrt{a})\int_K\left(\int_K D_G(\sqrt{a}k_2k_1k_2^*)d^*k_2\right) d^*k_1.\nonumber
\end{eqnarray} 
In this calculation we have used three ingredients. First, we have $\sqrt{k_2^*ak_2}=k_2^*\sqrt{a}k_2$ for all $k_2\in K$ by the spectral theorem. Second, the bi-unitary invariance of $ f_G^{(K)}$ yields $ f_G^{(K)}(k_1 \sqrt{k_2^*ak_2})= f_G^{(K)}(\sqrt{a})$ for all $k_1,k_2\in K$. And third, the $G$-adjoint-invariance of $D_G$ implies a $K$-invariance via the adjoint action such that $D_G(k_1 \sqrt{k_2^*ak_2})=D_G(k_1k_2^* \sqrt{a}k_2)=D_G(\sqrt{a}k_2k_1k_2^*)$. Furthermore $k_2^*k_1k_2=k\in K$ is also a Haar distributed unitary matrix. Hence we obtain Eq.~\eqref{deform-sv}.

The last statement~\eqref{Rel-deform} immediately follows from the fact that $f_G^{(K)}\in L^{1,K}(G)$ 
where the operators in the commutative diagram~\eqref{diagram} are defined.
\end{proof}

Starting from Theorem~\ref{thm:deform}, we directly obtain the corresponding joint density for the singular values of ensembles which are usually studied in the context of normal matrices, e.g. see Refs.~\cite{CZ:1998,TBAZW:2005,BK:2012}. Those normal matrices share the same eigenvalue densities of the form
\begin{equation}\label{normal}
\FEV(z)\propto|\Delta_n(z)|^2\prod_{j=1}^n\left(\omega(|z_j|^2)|\chi(z_j)|^2\right),
\end{equation}
where $\omega\in L^{1,n-1}_{[1,n]}(\mathbb{R}_+)$ and $\chi:\mathbb{C}\rightarrow\mathbb{C}$ an entire function such that \begin{equation}\label{cond-norm}
\int_{\mathbb{C}} |z|^{b}\omega(|z|^2)|\chi(z)|^2dz<\infty\ {\rm for\ all}\ b=0,\ldots,2n-2.
\end{equation}
The main difference between normal matrix ensembles and bi-unitarily invariant matrix ensembles is the level repulsion of the singular values. For normal matrices we commonly do not have no level repulsion while for bi-unitarily invariant matrices this is usually the case. Thus the singular values of bi-unitarily matrices spread much stronger than those of normal matrices.

There is a natural bi-unitarily invariant matrix model associated to the joint density~\eqref{normal} via Eq.~\eqref{Rel-deform} which is the content of our next result.

\begin{corollary}[$\FSV$ of $G$-adjoint-invariant Deformations]\label{cor:deform}\

 Consider the $G$-adjoint-invariant deformation
 \begin{equation}\label{norm-deform}
  f_G(g)=|\det\chi(g)|^2\ \mathcal{I}_\Omega^{-1}\mathcal{I}_A^{-1}\FSV^{(n)}([\omega];g)
\end{equation}
of the bi-unitarily invariant ensemble $\mathcal{I}_\Omega^{-1}\mathcal{I}_A^{-1}\FSV^{(n)}[\omega]$ with $\FSV^{(n)}[\omega]$ given by Eq.~\eqref{eq:PE-der} and $\det\chi(g)=\prod_{j=1}^n\chi(z_j(g))$ with $\chi$ an entire function as in Eq.~\eqref{cond-norm} and $z_j(g)$, $j=1,\ldots,n$, the complex eigenvalues of $g$. Then the joint density of the eigenvalues is given by Eq.~\eqref{normal} and the one of the singular values is
\begin{equation}\label{sv-norm-deform}
\FSV(a)=\left(\int_K |\det\chi(\sqrt{a}k)|^2 d^*k\right)\FSV^{(n)}([\omega];a).
\end{equation}
\end{corollary}

\begin{proof}
 Since $|\det\chi(g)|^2=\prod_{j=1}^n|\chi(z_j(g))|^2$ is $G$-adjoint-invariant by construction, we obtain Eq.~\eqref{normal} for the joint density for the eigenvalues because of Eq.~\eqref{deform-ev}. The joint density of the squared singular values follows from Eq.~\eqref{deform-sv}. The group integral over $K$ explicitly reads in the present case like the one in Eq.~\eqref{sv-norm-deform}, which concludes the proof.
\end{proof}

The construction of the deformed random matrix ensembles in Corollary~\ref{cor:deform} is far from being only academical. For example the elliptic Ginibre ensemble~\cite{Akemann:2001} falls into this class which was employed in the description  of three-dimensional QCD with chemical potential. Moreover, to demonstrate that this construction can also be made very explicit for some ensembles, let us state two examples.

\begin{examples}[$G$-adjoint-invariant deformations]\

 Let us choose a polynomial ensemble of derivative type with the function $\omega\in L^{1,n-1}_{[1,n]}(\mathbb{R}_+)$ which we want to deform. We only assume that it satisfies the integrability~\eqref{cond-norm} for the following two deformations. \\
 \begin{enumerate}[(a)]
 \item	The first deformation is $\chi(z)=e^{\alpha z/2}$ with $\alpha\in\mathbb{C}_*$. Then the density on $G$ is
 			\begin{equation}\label{g-def-1}
 			f_G(g)=\exp[{\rm Re}(\alpha \tr g)] f_G^{(n)}([\omega];g)
 			\end{equation}
 			and the joint density of the eigenvalues reads
 			\begin{equation}\label{ev-def-1}
 			\FEV(z)=\frac{C_{\rm sv}^{(n)}[\omega]\prod_{j=0}^{n-1}j!}{\pi^n}|\Delta_n(z)|^2\prod_{j=1}^n\left(\exp[{\rm Re}(\alpha z_j)]\omega(|z_j|^2)\right).
 			\end{equation}
 			For deriving the joint density of the squared singular values we employ the  Leutwyler-Smilga integral~\cite{LS-int,SW:2003},
 			\begin{equation}\label{LS-integral}
 			\int_K \exp[{\rm Re}(\alpha \tr \sqrt{a}k)]d^*k=\frac{\prod_{j=0}^{n-1}j!}{\alpha^{n(n-1)}}\frac{\det[(\alpha \sqrt{a_b})^{c-1}I_{c-1}(2\alpha\sqrt{a_b})]_{b,c=1,\ldots,n}}{\Delta_n(a)}
 			\end{equation}
 			with $I_\nu$ the modified Bessel functions of the first kind. Then we have
 			\begin{eqnarray}\label{sv-def-1}
 			\FSV(a)&=&\frac{C_{\rm sv}^{(n)}[\omega]\prod_{j=0}^{n-1}j!}{\alpha^{n(n-1)}}\det[(\alpha \sqrt{a_b})^{c-1}I_{c-1}(2\alpha\sqrt{a_b})]_{b,c=1,\ldots,n}\\
 			&&\times\det[(-a_k\partial_{a_k})^{j-1}\omega(a_k)]_{j,k=1,\hdots,n}.\nonumber
 			\end{eqnarray}
 			
 			This ensemble with $\omega(a)=a^{\nu/2}K_{\nu/2}(a)$ the modified Bessel function of the second kind, $\nu>0$, was recently considered in Refs.~\cite{AS-coupled-1,AS-coupled-2} where the singular values of a product of two coupled Gaussian distributed rectangular matrices were studied. For a product of more than two matrices drawn from Gaussian ensembles this coupling does not work due to loss of integrability. However when those matrices are drawn from Jacobi ensembles, see Eq.~\eqref{Jac-def}, the integration domain is compact and thus no integrability issues arise.

 			To understand what the deformation does with the spectrum let us sketch the limit $\alpha\to\infty$ with  $\alpha>0$.  For the Laguerre ensemble $\omega_{\rm Lag}(a)=a^\nu e^{-a}$ with $\nu > 0$, we can shift the term $\alpha$ away in the real parts of the eigenvalues $z$ which suppresses the level repulsion from the origin. This level repulsion is reflected in the term $a^\nu$ in the weight $\omega$ and carries over to the complex eigenvalues as $|\det z|^{2\nu}$. It is after the shift in $\alpha$ equal to $|\det(z+\alpha\eins_n)|^2=\prod_{j=1}^n((x_j+\alpha)^2+y_j^2)^\nu\overset{\alpha\gg1}{\approx}\alpha^{2\nu n}$. Note that the Vandermonde determinant is translation invariant which allows this shift.
 			
 			In contrast to the Gaussian case, one can also consider the deformation of the Jacobi ensemble $\omega_{\rm Jac}(a)=a^\nu (1-a)^{\mu-1}\Theta(1-a)$. Then, we expand about the contributing extremum $z^{(0)}=\eins_N$ as follows $z_j=(1-\delta r_j/\alpha)\exp[\imath\delta\varphi_j/\sqrt{\alpha}]$ yielding a decoupling of the spectrum into a Gaussian unitary ensemble whose eigenvalues are described by $\delta\varphi_j$ and the radial perturbations $\delta r_j$ become statistically independently, identically distributed random variables drawn from Gamma distributions. This behaviour can be expected for all bi-unitarily invariant ensembles with a compact support. The deformation shifts the spectrum to the utmost point with the largest real part. Since the boundary is of one dimension lower than the interior of the support we have a splitting of scales of the spectra into the radial and the angular part with a concentration on the boundary.  For the Cauchy-Lorentz ensemble~\eqref{CL-def} as well as for many other ensembles this deformation is not eligible due to integrability.
 			
 			As a conclusion, the deformation $\chi(z)=e^{\alpha z/2}$ can result in very different effects, ranging from suppressions of repulsions to elliptic deformations similar to the one of the elliptic Ginibre ensemble~\cite{Akemann:2001}.
 			
 \item	
 The second deformation we want to consider is $\chi(z)=(\alpha-z)^{\gamma/2}$ with $\gamma\in\mathbb{N}$ and $\alpha\in\mathbb{C}$. Again the density on $G$,
 			\begin{equation}\label{g-def-2}
 			f_G(g)=|\det(\alpha\eins_n-g)|^\gamma f_G^{(n)}([\omega];g)
 			\end{equation}
 			and on $Z$
 			\begin{equation}\label{ev-def-2}
 			\FEV(z)=\frac{C_{\rm sv}^{(n)}[\omega]\prod_{j=0}^{n-1}j!}{\pi^n}|\Delta_n(z)|^2\prod_{j=1}^n\left(|\alpha-z_j|^\gamma\omega(|z_j|^2)\right)
 			\end{equation}
 			are immediately given.
 			
 			For the singular value density we have to evaluate the group integral
 			\begin{equation}\label{group-int-def}
 			 J(a):=\int_K |\det(\alpha\eins_n- \sqrt{a}k)|^\gamma d^*k=\int_K (\det[|\alpha|^2\eins_n- ak])^\gamma (\det[\eins_n-k^*])^\gamma d^*k.
 			\end{equation}
 			The second equality is true because of the following calculation
 			\begin{eqnarray}
 			J(a)&=&|\alpha|^{2\gamma}\int_{K^3} \det(k_1k_2)^{-\gamma} \exp\left[\tr(k_1+k_2)-\frac{1}{\alpha}\tr k_1\sqrt{a}k-\frac{1}{\alpha^*}\tr k_2k^*\sqrt{a}\right]\nonumber\\
 			&&\times d^*k_2d^*k_1d^*k\nonumber\\
 			&=&|\alpha|^{2\gamma}\int_{K^3} \det(k_1k_2)^{-\gamma} \exp\left[\tr(k_1+k_2)-\frac{1}{|\alpha|^2}\tr k_1ak-\tr k_2k^*\right]\nonumber\\
 			&&\times d^*k_2d^*k_1d^*k.\label{pol-def-calc-a}
 			\end{eqnarray}
 			Thereby we have used in the first equality the identity $(\det h)^\gamma=\int_K (\det k')^{-\gamma}$ $\exp[\tr hk']d^*k'$ for any matrix $h\in G$. Note that all three integrals over the unitary group $K$ are compact, so that we can interchange the integrals as we want. The second equality follows from the invariance of the Leutwyler-Smilga integral~\eqref{LS-integral} which is in the present case the integral over $k$.  The Leutwyler-Smilga integral only depends on the product $k_1ak_2/|\alpha|^2$, cf. Ref.~\cite{SW:2003}. Hence we can rewrite the expression into the second line of Eq.~\eqref{pol-def-calc-a}.
The integration over $k_1$ and $k_2$ yields the right hand side of Eq.~\eqref{group-int-def}.
 			
 			In the next step we diagonalize $k={k'}^* \Phi k'$ with $\Phi=\diag(e^{\imath\varphi_1},\ldots,e^{\imath\varphi_n})\in[{\rm U}(1)]^n$ a diagonal matrix of phases  and $k'\in {\rm U}(n)/[{\rm U}(1)]^n$. This yields a change of the measure as $d^* k=(\prod_{j=1}^n 2\pi^j/j!) |\Delta_n(\Phi)|^2 d^*\Phi d^*k'$ with $d^*\Phi$ as in Eq.~\eqref{I-Z-alt}. The integral over $k'$ is the Itzykson-Zuber-Harish-Chandra-like group integral~\cite{group-int-1,group-int-2}
 			\begin{eqnarray}\label{IZHC-like}
 			&&\int_{{\rm U}(n)/[{\rm U}(1)]^n} (\det[|\alpha|^2\eins_n- a{k'}^*\Phi k'])^\gamma dk'\\
 			&=&\left(\prod_{j=0}^{n-1}\frac{(-1)^j j!(\gamma+j)!}{(\gamma+n-1)!|\alpha|^{2j}}\right)\frac{\det[(|\alpha|^2- a_b e^{\imath\varphi_c})^{\gamma+n-1}]_{b,c=1,\ldots,n}}{\Delta_n(a)\Delta_n(\Phi)}.\nonumber
 			\end{eqnarray}
 			Applying the Andr\'eief identity~\cite{Andreief} the integral~\eqref{group-int-def} simplifies to
 			\begin{eqnarray}
 			J(a)&=&\left(\prod_{j=0}^{n-1}\frac{(-\pi)^j (\gamma+j)!}{(\gamma+n-1)!|\alpha|^{2j}}\right)\label{pol-def-calc-b}\\
 			&&\times\frac{\det[\int_{-\pi}^\pi (|\alpha|^2- a_b e^{\imath\varphi})^{\gamma+n-1}(1-e^{-\imath\varphi})^\gamma e^{-\imath(c-1)\varphi}d\varphi]_{b,c=1,\ldots,n}}{\Delta_n(a)}\nonumber\\
 			&=&\left(\prod_{j=0}^{n-1}\frac{2\pi^{j+1} (\gamma+j)!}{(\gamma+n-1)!}\right)\frac{1}{\Delta_n(a)}\nonumber\\
 			&&\times\det\left[\sum_{l=0}^\gamma\frac{\gamma!(\gamma+n-1)!}{l!(\gamma-l)!(n-c+l)!(\gamma+c-l-1)!}|\alpha|^{2l}a_b^{\gamma-l+c-1}\right]_{b,c=1,\ldots,n}.\nonumber
 			\end{eqnarray}
 			Hence we arrive at the singular value density
 			\begin{eqnarray}\label{sv-def-2}
 			\FSV(a)&=&C_{\rm sv}^{(n)}[\omega]\left(\prod_{j=0}^{n-1}\frac{2\pi^{j+1}}{j!}\right)\det[(-a_k\partial_{a_k})^{j-1}\omega(a_k)]_{j,k=1,\hdots,n}\\
 			&&\times\det\left[a_b^{\gamma+c-1}\,_{2}F_1\left(-\gamma,1-\gamma-c;n-c+1\biggl|\frac{|\alpha|^2}{a_b}\right)\right]_{b,c=1,\ldots,n}\nonumber
 			\end{eqnarray}
 			with $\,_{2}F_1$ the ordinary hypergeometric function.
 			
 			We  expect that the model~\eqref{g-def-2} as well as  joint densities~\eqref{ev-def-2} and \eqref{sv-def-2} can be analytically continued to real $\gamma>-1$ because the hypergeometric function is also defined for real indices. We will neither discuss nor derive this claim and let it stand as a conjecture.
 			
 			Another point we want to underline is that this kind of deformation allows to open a hole in the complex spectrum 
 			 for a variety of polynomial ensembles of derivative type, especially of Meijer G-ensembles, as the parameter $\gamma$ tends to infinity. Thus this ensembles creates phase transitions which can now be analyzed simultaneously at the level of eigenvalues and at the level of singular values with the help of our approach.
 \end{enumerate}
\end{examples}

Both examples above give rise to determinantal point process in their singular value as well as eigenvalue statistics. This can be readily seen by their explicit expressions~\eqref{ev-def-1}, \eqref{sv-def-1}, \eqref{ev-def-2}, and \eqref{sv-def-2} and general calculations~\cite{Borodin,KG-det,Akemann-book} for ensembles built out of bi-orthogonal functions. 
We recall that determinantal point processes are generally only an algebraic statement which is based on the fact that the two determinants involved in the joint densities are multi-linear and skew-symmetric. We keep it by these statements and will not go into the details of the statistics for these two particular examples since it will exceed the present discussion.

Nonetheless, to underline that our approach also yields new insights into the direct relation between the eigenvalue and singular value statistics, we will study the kernels of  bi-unitarily invariant matrix ensembles without a deformation in Section~\ref{sec:Implications}. In doing so we restrict ourselves to polynomial ensembles of derivative type.

% ****************************************************************************************************

\section{Implications for the Kernels}
\label{sec:Implications}

Let us consider a polynomial ensemble $\FSV^{(n)}[\omega]\in L^{1,{\rm SV}}(A)$  of derivative type, see Eq.~\eqref{eq:PE-der}. We assume that it results from a bi-unitarily invariant matrix ensemble on $G$ via the maps $f_G=\mathcal{I}_\Omega^{-1}\mathcal{I}_A^{-1}\FSV^{(n)}[\omega]\in L^{1,K}(G)$, cf. Eq.~\eqref{FG-PE} for an explicit representation of $f_G$. Moreover we assume that the function $\omega\in L^{1,n-1}_{]s_{\min},s_{\max}[}(\mathbb{R}_+)$ is positive and that its Mellin transform $\mathcal{M}\omega$ is analytic on a slightly larger strip ${]}s_{\min},s_{\max}{[}+\imath\mathbb{R}\supset [1,n]+\imath\mathbb{R}$ than originally required, i.e. $s_{\min}<1$ and $s_{\max}>n$. The latter additional assumption resolves some technical problems when choosing the contours for particular representations of the kernels. However we would expect that it can be dropped by choosing other contours than those we employ. 

The first assumption about the positivity of $\omega$ is also only technical and not really of relevance. First we want to define the kernel of the eigenvalue statistics in the standard way which is symmetric and involves the square root of $\omega$; see Eq.~\eqref{ker-ev} below. When defining the kernel non-symmetrically we can avoid this square root and, hence, are allowed to drop this assumption. Second the positivity is also used to exclude poles of the function $1/\mathcal{M}\omega$ on the interval $[1,n]$, so~that we have a single contour encircling poles coming from specific Gamma functions. Also here one can think of relaxing the requirement to the case that $\mathcal{M}\omega$ is non-vanishing at the points $s=1,\ldots,n$ which is a direct consequence of the fact that $\FEV^{(n)}[\omega]$ as well as $\FSV^{(n)}[\omega]$ are densities and thus normalizable. Then the contours involved in the calculation have to encircle the desired poles close enough.

We structure this section into three parts. 
First we summarize the results for the eigenvalues in Lemma~\ref{lem:pol-ev}. Analogously we summarize the results for the squared singular values in Lemma~\ref{lem:pol-sv}. Finally we relate both statistics in Theorem~\ref{thm:rel-pol-SEV}, which {constitutes the main result of this section.

\begin{lemma}[Eigenvalue Kernel]\label{lem:pol-ev}\

Consider the normalized joint density
				\begin{equation}\label{pol-der-ev}
				\FEV^{(n)}([\omega];z)=C_{\rm ev}^{(n)}[\omega]|\Delta_n(z)|^2\prod_{j=1}^n \omega(|z_j|^2).
				\end{equation}	
				  of the eigenvalues $z\in Z$ of a bi-unitarily invariant matrix ensemble corresponding to a polynomial ensemble of derivative type with the positive function $\omega\in L^{1,n-1}_{]s_{\min},s_{\max}[}(\mathbb{R}_+)$ and $[1,n]\subset {]s_{\min},s_{\max}[}\subset\mathbb{R}$. Then the normalization constant is
				 \begin{equation}\label{const-ev}
				 C_{\rm ev}^{(n)}[\omega]=\frac{1}{n!\pi^n}\,\frac{1}{\prod_{j=1}^{n}\mathcal{M}\omega(j)}
				 \end{equation}
				 and $\FEV$ gives rise to a determinantal point process, i.e.
\begin{equation}\label{deter-ev}
 \FEV^{(n)}([\omega];z)=\frac{1}{n!}\det\Big[K_{\rm ev}^{(n)}([\omega];z_b,\bar{z}_c)\Big]_{b,c=1,\ldots,n},
\end{equation}
with the kernel
\begin{equation}\label{ker-ev}
K_{\rm ev}^{(n)}([\omega];z_b,\bar{z}_c)=\sqrt{\omega(|z_b|^2)\omega(|z_c|^2)}\sum_{j=0}^{n-1}\frac{(z_b\bar{z}_c)^j}{\pi\mathcal{M}\omega(j+1)}.
\end{equation}
				Hence the corresponding orthogonal polynomials are the monomials $z\mapsto z^j$ with the normalization constants
				\begin{equation}\label{orth-ev}
				\int_{\mathbb{C}} z^i\bar{z}^j\omega(|z|^2)dz=\pi\mathcal{M}\omega(j+1)\delta_{ij}.
				\end{equation}
				The $k$-point correlation function, $k=1,\ldots,n$, is given by
\begin{equation}\label{k-point-ev}
 R^{(n,k)}_{\rm ev}([\omega];z)=\det[K_{\rm ev}^{(n)}([\omega];z_b,\bar{z}_c)]_{b,c=1,\ldots,k} \,.
\end{equation}
In particular, the normalized level density is
\begin{equation}\label{level-ev}
 \varrho^{(n)}_{\rm ev}([\omega];z)=\frac{1}{n} K_{\rm ev}^{(n)}([\omega];z,\bar{z})]=\frac{\omega(|z|^2)}{n}\sum_{j=0}^{n-1}\frac{|z|^{2j}}{\pi\mathcal{M}\omega(j+1)}.
\end{equation}
\end{lemma}

Note that we normalize the $k$-point correlation functions via the recursion
\begin{equation}\label{norm-k-point-ev}
 \int_{\mathbb{C}} R_{\rm ev}^{(n,k)}([\omega];z_1,\ldots,z_{k-1},z_k) dz_k=(n-k+1)R_{\rm ev}^{(n,k-1)}([\omega];z_1,\ldots,z_{k-1}).
\end{equation}
Moreover we recover the joint probability density $\FEV^{(n)}([\omega];z)=R_{\rm ev}^{(n,n)}([\omega];z)/n!$ for $k=n$.

\begin{proof}
 The normalization constant can be calculated in a straightforward way,
 \begin{eqnarray}\label{ev-calc-a}
 \frac{1}{C_{\rm ev}^{(n)}[\omega]}&=&\int_Z |\Delta_n(z)|^2\left(\prod_{j=1}^n \omega(|z_j|^2)\right)dz\\
 &=&\pi^n \int_A {\rm Perm}[a_b^{c-1}]_{b,c=1,\ldots,n}\left(\prod_{j=1}^n \omega(a_j)\right)da.
 \end{eqnarray}
 In the second line we have integrated over the phases of the eigenvalues $z$, see Eq.~\eqref{eq:VandermondeIntegral}. Expanding this permanent yields $n!$ times the same term for symmetry reasons,
 while the remaining integral factorizes. The integral over each single $a_j$ is a Mellin transform of $\omega$, which proves Eq.~\eqref{const-ev}.
 
 The statement that the monomials are the orthogonal polynomials is obvious since apart from the Vandermonde determinants the joint density has no phase dependence. In particular the orthogonality~\eqref{orth-ev} follows from the integration over the phase while the integration over the radius is the Mellin transformation. Analogously one can show
 \begin{eqnarray}
 \int_{\mathbb{C}} K_{\rm ev}^{(n)}([\omega];z_1,\bar{z}_2)K_{\rm ev}^{(n)}([\omega];z_2,\bar{z}_3)dz_2&=&K_{\rm ev}^{(n)}([\omega];z_1,\bar{z}_3)\ {\rm and}\nonumber\\
  \int_{\mathbb{C}}K_{\rm ev}^{(n)}([\omega];z,\bar{z})dz&=&n\label{ev-calc-b} \,,
 \end{eqnarray}
which shows that the kernel is the correct one for the $k$-point correlation function~\eqref{k-point-ev}.
 The only thing to be checked is the determinantal point process property~\eqref{deter-ev} 
 which follows from general discussions~\cite{CZ:1998,Borodin}, especially one can readily rewrite
 \begin{eqnarray}\label{ev-calc-c}
 \FEV^{(n)}([\omega];z)=\frac{1}{n!}\det\biggl[\frac{z_b^{c-1}\sqrt{\omega(|z_b|^2)}}{\pi\mathcal{M}\omega(c+1)}\biggl]_{b,c=1,\ldots,n}\det[\bar{z}_b^{c-1}\sqrt{\omega(|z_b|^2)}]_{b,c=1,\ldots,n},
 \end{eqnarray}
 where we pushed parts of the density into the two Vandermonde determinants. 
 The product rule for determinants, i.e. $\det BC=\det B \det C$ for two square matrices $B$ and $C$, 
 yields the claim~\eqref{deter-ev}.
\end{proof}

Let us emphasize once again the following two things. In the proof above we did not use the assumption
that the Mellin transformation of $\omega$ exists on a slightly larger interval than $[1,n]$.
We only need this requirement in the following lemma and Theorem~\ref{thm:rel-pol-SEV}. 
Also, the positivity of $\omega$ can be dropped when choosing the kernel 
$\omega(|z_b|^2)\sum_{j=0}^{n-1}(z_b\bar{z}_c)^j/(\pi\mathcal{M}\omega(j+1))$ 
instead which still satisfies Eq.~\eqref{ev-calc-b}.

\begin{lemma}[Singular Value Statistics]\label{lem:pol-sv}\

Consider the normalized joint density
\begin{equation}\label{pol-der-sv}
 \FSV^{(n)}([\omega];a)=C_{\rm sv}^{(n)}[\omega]\Delta_n(a) \, \det[(-a_k\partial_{a_k})^{j-1}\omega(a_k)]_{j,k=1,\hdots,n}
 \end{equation}
 of the squared singular values $a\in A$ corresponding to the joint density of Lemma~\ref{lem:pol-ev} via $\FSV^{(n)}[\omega]=\mathcal{R}^{-1}\FEV^{(n)}[\omega]$. The normalization constant is equal to
				 \begin{equation}\label{const-sv}
				 C_{\rm sv}^{(n)}[\omega]=\frac{1}{\prod_{j=1}^{n}j!}\,\frac{1}{\prod_{j=1}^{n}\mathcal{M}\omega(j)}
				 \end{equation}
				 The joint density~\eqref{pol-der-sv} gives rise to a determinantal point process
\begin{equation}\label{deter-sv}
 \FSV^{(n)}([\omega];a)=\frac{1}{n!}\det[K_{\rm sv}^{(n)}([\omega];a_b,a_c)]_{b,c=1,\ldots,n},
\end{equation}
with the kernel
\begin{equation}\label{ker-sv-a}
K_{\rm sv}^{(n)}([\omega];a_b,a_c)=\sum_{j=0}^{n-1}p_j(a_b)q_j(a_c).
\end{equation}
Here we employed the following polynomials in monic normalization
\begin{eqnarray}\label{pol-sv}
p_{l}([\omega];a)&=&\sum_{j=0}^l(-1)^{l-j}\frac{l!\mathcal{M}\omega(l+1)}{j!(l-j)!\mathcal{M}\omega(j+1)}a^j\\
&=&l! \mathcal{M}\omega(l+1)\oint_{\mathcal{C}}\frac{\Gamma(t-l-1)}{\Gamma(t)\mathcal{M}\omega(t)}a^{t-1}\frac{dt}{2\pi\imath},\nonumber
\end{eqnarray}
 $l=0,\ldots,n-1$, where the closed contour $ \mathcal{C}$ encircles the interval $[1,n]$ close enough such that $1/\mathcal{M}\omega(t+1)$ has no poles in the interior of the contour
and $s_{\min} < \re t < s_{\max}$ for all $t \in \mathcal{C}$.
 The functions bi-orthogonal to these polynomials are 
\begin{eqnarray}\label{func-sv}
q_{l}([\omega];a)&=&\frac{1}{l!\mathcal{M}\omega(l+1)}\left[\prod_{j=1}^l\left(-a\partial_a-j\right)\right]\omega(a)\\
&=&\frac{1}{l!\mathcal{M}\omega(l+1)}\partial_a^l\left[(-a)^l\omega(a)\right]\nonumber\\
&=&\frac{1}{l!\mathcal{M}\omega(l+1)}\lim_{\epsilon\to0}\int_{-\infty}^\infty\frac{\pi^2\cos(\epsilon  s)}{\pi^2-4\epsilon^2 s^2}\frac{\Gamma(s_0+\imath s)\mathcal{M}\omega(s_0+\imath s)}{\Gamma(s_0+\imath s-l)}%\nonumber\\&&\times 
a^{-s_0-\imath s}\frac{ds}{2\pi}\nonumber
\end{eqnarray}
with $s_0\in {]s_{\min},1[}$ and $s_0>0$ chosen such that $s_0<{\rm Re}\, t$ for all $t\in\mathcal{C}$. For $l=0$ we omit the product of the derivatives in the first line of Eq.~\eqref{func-sv}. The two sets of functions satisfy the bi-orthogonality relation
				\begin{equation}\label{orth-sv}
				\int_{0}^\infty  p_i([\omega];a)q_j([\omega];a)da=\delta_{ij}.
				\end{equation}
				The $k$-point correlation function and the level density are
\begin{equation}\label{k-point-sv}
 R^{(n,k)}_{\rm sv}([\omega];a)=\det[K_{\rm sv}^{(n)}([\omega];a_b,a_c)]_{b,c=1,\ldots,k}
\end{equation}
and
\begin{equation}\label{level-sv}
 \varrho^{(n)}_{\rm sv}([\omega];a)=\frac{1}{n} K_{\rm sv}^{(n)}([\omega];a,a)]=\frac{1}{n}\sum_{j=0}^{n-1}p_j([\omega];a_b)q_j([\omega];a_c),
\end{equation}
respectively.

In the case that  $\omega\in L^{1,n}_{]s_{\min},s_{\max}[}(\mathbb{R}_+)$ and $s_{\max}>n+1$ we have the alternative representation
\begin{equation}\label{ker-sv-b}
K_{\rm sv}^{(n)}([\omega];a_b,a_c)=-n\frac{\mathcal{M}\omega(n+1)}{\mathcal{M}\omega(n)}\int_0^1p_{n-1}([\omega];xa_b)q_n([\omega];xa_c)dx.
\end{equation}
\end{lemma}

The normalization of the $k$-point correlation functions~\eqref{k-point-sv} is given similarly to Eq.~\eqref{norm-k-point-ev}. The calculation for the singular value statistics is a little bit more involved than that for the eigenvalues. However we will pursue the standard approaches used in the calculus of bi-orthogonal polynomials~\cite{CZ:1998,Borodin,KG-det}.

\begin{proof}
Considering the two determinants in the definition~\eqref{pol-der-sv} of the density $\FSV^{(n)}[\omega]$, 
we immediately recognize that
\begin{eqnarray}\label{sv-calc-a}
\underset{j=0,\ldots,n-1}{\rm span}\{a^{j}\}&=&\underset{j=0,\ldots,n-1}{\rm span}\{p_j([\omega];a)\} \quad {\rm and}\\ \underset{j=0,\ldots,n-1}{\rm span}\{(-a\partial_a)^{j}\omega(a)\}&=&\underset{j=0,\ldots,n-1}{\rm span}\{q_j([\omega];a)\}.\nonumber
\end{eqnarray}
Hence the polynomials $p_l$ and the functions $q_l$ as well as the determinantal form~\eqref{deter-sv} of the density $\FSV^{(n)}[\omega]$ can be constructed via linear combination of the rows in the two determinants and the factorization rule of the determinants, cf. the end of the proof of Lemma~\ref{lem:pol-ev}. Moreover the function $\omega$ is positive such that $\mathcal{M}\omega(s)>0$ for all $s\in {]s_{\min},s_{\max}[}$. Since $\mathcal{M}\omega(s)$ is holomorphic and, thus, continuous on the complex strip ${]s_{\min},s_{\max}[}+\imath\mathbb{R}\subset\mathbb{C}$, there is an open neighbourhood $U\subset\mathbb{C}$ such that
$[1,n]\subset U\subset{]s_{\min},s_{\max}[}+\imath\mathbb{R}$ and $|\mathcal{M}\omega(s)|>0$ for all $s\in U$, i.e. $1/\mathcal{M}\omega$ is holomorphic on $U$, too.
Choosing a  contour $\mathcal{C}\subset U$ encircling the real interval $[1,n]$ we obtain via the residue theorem the contour integral~\eqref{pol-sv} for the polynomials. Moreover, similarly as in \eqref{eq:mellin-link}, we~have
\begin{align}
\label{eq:mellin-link-2}
\mellin([\partial_a^l ((-a)^l \omega(a))];t) = \frac{\Gamma(t)\mathcal{M}\omega(t)}{\Gamma(t-l)} \,.
\end{align}
Thus, the contour integral representation \eqref{func-sv} for the function $q_l$ follows 
from the Mellin inversion formula~\eqref{eq:MellinInversion}.

The biorthogonality can be seen by the straightforward calculation
\begin{eqnarray}
\int_{0}^\infty  p_i([\omega];a)q_j([\omega];a)da&=&  i! \mathcal{M}\omega(i+1)\oint_{\mathcal{C}}\frac{\Gamma(t-i-1)}{\Gamma(t)\mathcal{M}\omega(t)}\left(\int_{0}^\infty a^{t-1}q_j([\omega];a)da\right)\frac{dt}{2\pi\imath}\nonumber\\
&=&i! \mathcal{M}\omega(i+1)\oint_{\mathcal{C}}\frac{\Gamma(t-i-1)}{\Gamma(t)\mathcal{M}\omega(t)}\mathcal{M}q_j([\omega];t)\frac{dt}{2\pi\imath}\nonumber\\
&=&\frac{i! \mathcal{M}\omega(i+1)}{j! \mathcal{M}\omega(j+1)}\oint_{\mathcal{C}}\frac{\Gamma(t-i-1)}{\Gamma(t-j)}\frac{dt}{2\pi\imath}.\label{sv-calc-b}
\end{eqnarray}
In the first line we may interchange the integrals over $t$ and $a$ because the functions are Lebesgue integrable. In particular $|\Gamma(t-i-1)/(\Gamma(t)\mathcal{M}\omega(t))|$ is bounded from above on the contour $\mathcal{C}$ and $a^{{\rm Re}\, t}$ is bounded from above by the function $a^{t_{\max}}+a^{t_{\min}}$ with $t_{\max},t_{\min} \in {]s_{\min},s_{\max}[}$ the maximal and minimal real part of the contour $\mathcal{C}$, respectively. Since $\omega\in L^{1,n-1}_{]s_{\min},s_{\max}[}(\mathbb{R}_+)$ and thus also $(a\partial)^j\omega(a)\in L^{1}_{]s_{\min},s_{\max}[}(\mathbb{R}_+)$ for all $j=0,\ldots,n-1$ we have $q_j\in L^{1}_{]s_{\min},s_{\max}[}(\mathbb{R}_+)$. Therefore the Mellin trans\-form of $q_j$, $j=0,\ldots,n-1$, is well-defined on the complex strip ${]s_{\min},s_{\max}[}+\imath\mathbb{R}\subset\mathbb{C}$. The third line of Eq.~\eqref{sv-calc-b}
follows from Eq.~\eqref{eq:mellin-link-2}.

The remaining contour integral over $t$ vanishes if $j>i$ because the integrand does not have any pole. On the other hand for $j<i$ the integrand drops of at least like $t^{-2}$ at infinity. Hence, when deforming the contour to a circle and taking the radius of this circle to infinity, the integral becomes zero as well. For $j=i$ we obtain the normalization of the biorthogonal functions, i.e.\@ unity.

The bi-orthogonal structure immediately implies an identity similar to Eq.~\eqref{ev-calc-b}.
The claims about the determinantal point process, the $k$-point correlation function and the level density 
are direct consequences of this.

Regarding the alternative representation~\eqref{ker-sv-b} we follow the ideas in Ref.~\cite{KZ:2014}. Let us consider the identity
\begin{equation}\label{sum-identity}
\sum_{j=0}^{n-1}\frac{\Gamma[t-j-1]\Gamma[s_0+\imath s]}{\Gamma[t]\Gamma[s_0+\imath s-j]}=\frac{1}{s_0+\imath s-t}\left[\frac{\Gamma[t-n]\Gamma[s_0+\imath s]}{\Gamma[t]\Gamma[s_0+\imath s-n]}-1\right]
\end{equation}
which can be readily proven by induction. Expressing the polynomials $p_j[\omega]$ and the functions $q_j[\omega]$ in the sum~\eqref{ker-sv-a} by their integral representations we can interchange the sum and the integrals because the sum is finite.
Also, in the integral representation \eqref{func-sv} for $q_j[\omega]$,
we may replace $s$ with $s - \imath s_0$ in the regularizing function $\pi^2\cos(\epsilon s)/(\pi^2-4\epsilon^2s^2)$; 
compare the proof of Lemma \ref{lem:Mel-Inv}.
After that, we~can apply the identity~\eqref{sum-identity} leading to
\begin{eqnarray}
K_{\rm sv}^{(n)}([\omega];a_b,a_c)&=&\lim_{\epsilon\to0}\int_{-\infty}^\infty\frac{\pi^2\cos(\epsilon (s-\imath s_0))}{\pi^2-4\epsilon^2 (s-\imath s_0)^2}\Big(\oint_{\mathcal{C}}\left[\frac{\Gamma[t-n]\Gamma[s_0+\imath s]}{\Gamma[t]\Gamma[s_0+\imath s-n]}-1\right]\nonumber\\
&&\times\frac{\mathcal{M}\omega(s_0+\imath s)}{\mathcal{M}\omega(t)}\frac{a_b^{t-1}a_c^{-s_0-\imath s}}{s_0+\imath s-t}\frac{dt}{2\pi\imath}\Big)\frac{ds}{2\pi}.\label{sv-calc-c}
\end{eqnarray}
% Note that we cannot interchange the integrals now because they do not factorize anymore and we have to think about uniform convergence in % the auxiliary variable $\epsilon\to0$. 

Due to the choice of the contours the second term on the right hand side of Eq.~\eqref{sum-identity} vanishes under the $t$-integral because the contour encircles a region where the integrand is holomorphic. The fraction $1/(s_0+\imath s-t)$ can be rewritten as
\begin{eqnarray}
\frac{1}{s_0+\imath s-t}=-\int_0^1 x^{t-s_0-\imath s-1} dx\label{sv-calc-d}
\end{eqnarray}
because $s_0<{\rm Re}\, t$. 
The compact integrals over $x$ and $t$ can be interchanged without any problems, 
so that we can identify the polynomial $p_{n-1}([\omega],xa_b)$. 
Moreover, recalling Eq.~\eqref{eq:mellin-link-2},
even the integrals over $s$ and $x$ can be interchanged 
(but we have to keep the limit $\epsilon$ in front of all integrals) 
because the integrand is Lebesgue integrable over $s$ and $x$.
Due to the regularizing function depending on $\epsilon$ 
we obtain another compact integral over $a'$, 
see the proof of Lemma~\ref{lem:Mel-Inv},
\begin{eqnarray}
K_{\rm sv}^{(n)}([\omega];a_b,a_c)&=&\lim_{\epsilon\to0}\int_0^1 \Big(\int_{-1}^{1}\cos\left(\frac{\pi y'}{2}\right)q_{n}([\omega];xa_ce^{\epsilon y'})\frac{\pi dy'}{4}\Big)\nonumber\\
&&\times p_{n-1}([\omega];xa_b)dx,\label{sv-calc-e}
\end{eqnarray}
where we already identified the function $q_n[\omega]$. For $\epsilon$ small enough the integrand becomes bounded in $y' \in [-1,1]$ and $x \in [0,1]$. This allows us to pull the limit $\epsilon\to0$ into the integrals, which concludes the proof.
\end{proof}

Note that the singular value as well as eigenvalue statistics depend only on the ratio $\omega(a)/\mathcal{M}\omega(j)$ for $a\in A$ and $j=1,\ldots,n$. This property builds the bridge between the eigenvalue and the singular value statistics stated in Theorem~\ref{thm:rel-pol-SEV}. But before coming to this theorem let us emphasize that there is an even simpler contour integral expression for the polynomials~\eqref{pol-sv} in the case that there is an $r_0\in\mathbb{R}_+$ such that the Laurent series
\begin{equation}\label{series-def}
 Q([\omega];z):=\sum_{j=-\infty}^{\infty}\frac{(r_0z)^j}{\mathcal{M}\omega(j+1)}
\end{equation}
 has a radius of convergence of $1+\epsilon$ with $\epsilon>0$. For this series we do not need that $\mathcal{M}\omega(j+1)<\infty$ for all $j\in\mathbb{Z}$ because we interpret $1/\mathcal{M}\omega(j+1)=0$ when $\mathcal{M}\omega(j+1)=\infty$.

\begin{corollary}[Simplified Formula for the Polynomials]\label{cor:simpl-pol}\

 Assuming that there is an $r_0\in\mathbb{R}_+$ such that the function~\eqref{series-def} is holomorphic in an open neighborhood of the unit circle in $\mathbb{C}$. Then the polynomials~\eqref{pol-sv} simplify to
 \begin{equation}\label{simpl-pol}
  p_l([\omega];a)= \frac{\int_{-\pi}^\pi (a e^{\imath\varphi}-r_0)^l Q([\omega];e^{-\imath\varphi})d\varphi}{\int_{-\pi}^\pi e^{\imath l\varphi} Q([\omega];e^{-\imath\varphi})d\varphi},\ l=0,\ldots,n-1.
 \end{equation}
\end{corollary}

\begin{proof}
 We note that
 \begin{equation}\label{simpl-calc-a}
 \int_{-\pi}^\pi Q([\omega];e^{-\imath\varphi}) e^{\imath j\varphi} \frac{d\varphi}{2\pi}=\frac{r_0^j}{\mathcal{M}\omega(j+1)}
 \end{equation}
since $Q[\omega]$ is holomorphic in an open neighborhood of the unit circle. This identity can be plugged into the first line of Eq.~\eqref{pol-sv}. Then we have
 \begin{equation}\label{simpl-calc-b}
 p_l([\omega];a)= \sum_{j=0}^l(-1)^{l-j}\frac{l!\mathcal{M}\omega(l+1)}{j!(l-j)!}\left(\frac{a}{r_0}\right)^j\int_{-\pi}^\pi Q([\omega];e^{-\imath\varphi}) e^{\imath j\varphi} \frac{d\varphi}{2\pi}.
 \end{equation}
 Because the sum is finite we can interchange it with the compact contour integral and recognize a binomial sum. This yields the integral in the numerator of Eq.~\eqref{simpl-pol}. Due to the monic normalization we can fix the normalization which yields the integral in the denominator.
\end{proof}

\begin{remark}[Classical Orthogonal Polynomials]\

 The identity~\eqref{simpl-pol} makes contact to many known expressions of classical orthogonal polynomials. For example for the Ginibre ensemble~\eqref{Gin-def} with $\nu\in\mathbb{N}$ we have
 \begin{equation}\label{Lag-pol}
 Q([\omega_{\rm Lag}],z)=\sum_{j=-\infty}^\infty\frac{(r_0 z)^{j}}{\Gamma(\nu+j+1)}=\frac{e^{r_0z}}{(r_0z)^{\nu}},
 \end{equation}
 yielding with Eq.~\eqref{simpl-pol} the Laguerre polynomials. For the Jacobi polynomials~\eqref{Jac-def} the function~\eqref{series-def} reduces to
 \begin{equation}\label{Jac-pol}
 Q([\omega_{\rm Jac}],z)=\sum_{j=-\infty}^\infty\frac{\Gamma(\mu+\nu+j+1)(r_0 z)^{j}}{\Gamma(\mu)\Gamma(\nu+j+1)}=\frac{\mu}{(1-r_0z)^{\mu+1}(r_0z)^{\nu}},
 \end{equation}
 corresponding to the shifted Jacobi polynomials, and for the Cauchy-Laguerre ensemble it is
 \begin{equation}\label{CL-pol}
 Q([\omega_{\rm CL}],z)=\sum_{j=-\infty}^\infty\frac{\Gamma(\nu+\mu+1)(r_0 z)^{j}}{\Gamma(\mu-j)\Gamma(\nu+j+1)}=(\nu+\mu)\frac{(1+r_0z)^{\mu+\nu-1}}{(r_0z)^{\nu}},
 \end{equation}
both only for $\nu\in\mathbb{N}$. Indeed those relations can be found for other Meijer G-ensembles, too. Those simplifications were also found for certain ensembles with the help of the supersymmetry method~\cite{VKG-SUSY,Kieburg:2015}.

Moreover one can also consider the Muttalib-Borodin ensemble~\eqref{MB-def} or its limits~\eqref{ln-def} and \eqref{e-def}. Also for those ensembles the Laurent series~\eqref{series-def} exist. However we will omit them here and proceed with the discussion.
\end{remark}

The proof of Corollary~\ref{cor:simpl-pol} already outlines the main idea we are pursuing to derive the relations between the singular value and eigenvalue statistics. The signi\-ficant difference between Corollary~\ref{cor:simpl-pol} and the following theorem is that we do not have to assume that the Laurent series~\eqref{series-def} exists. This theorem relates the kernels and bi-orthogonal function of the eigenvalues and squared singular values. Thus  any statistical quantity for the singular values can be expressed in terms of the kernel~\eqref{ker-ev} describing the eigenvalue statistics.

\begin{theorem}[Relation between the Kernels]\label{thm:rel-pol-SEV}\

 We consider the same ensemble as in the Lemmas~\ref{lem:pol-ev} and \ref{lem:pol-sv} where $\omega$ is, additionally, $n$-times continuous differentiable and $[1,n+1]\subset {]s_{\min},s_{\max}[}$. Then, 
 \begin{enumerate}[(a)]
 \item	the polynomials are
 			\begin{equation}\label{pol-rel}
 			 p_l([\omega];a)=\int_0^\infty\left(\int_{-\pi}^\pi (ae^{\imath\varphi}-r^2)^lK_{\rm ev}^{(n)}([\omega];r,re^{-\imath\varphi})d\varphi\right)rdr,
 			\end{equation}
 			$l=0,\ldots,n-1$;
 \item	the functions bi-orthogonal to these polynomials are
 			\begin{equation}\label{func-rel}
 			 q_l([\omega];a)=\frac{1}{2l!}\left(-\partial_a\right)^l\left(\int_{-\pi}^\pi e^{\imath l \varphi}K_{\rm ev}^{(n)}([\omega];\sqrt{a},\sqrt{a}e^{-\imath\varphi})d\varphi\right),
 			\end{equation}
 			$l=0,\ldots,n$;
 \item	the kernel is
 			\begin{eqnarray}\label{ker-rel}
 			&&K_{\rm sv}^{(n)}([\omega];a_b,a_c)\\
 			&=&\frac{1}{2(n-1)!}\partial_{a_c}^n\left[\int_0^{a_c}\left(\int_{-\pi}^\pi K_{\rm ev}^{(n)}([\omega];\sqrt{x},\sqrt{x}e^{-\imath\varphi} )(a_c-a_be^{\imath\varphi})^{n-1}d\varphi\right)dx\right]\nonumber\\
 			&=&-\frac{1}{2(n-1)!}\partial_{a_c}^n\left[\int_{a_c}^\infty\left(\int_{-\pi}^\pi K_{\rm ev}^{(n)}([\omega];\sqrt{x},\sqrt{x}e^{-\imath\varphi} )(a_c-a_be^{\imath\varphi})^{n-1}d\varphi\right)dx\right].\nonumber
 			\end{eqnarray}
 \end{enumerate}
\end{theorem}

\begin{proof}
The proof is based on the following two identities
\begin{equation}\label{rel-pol-calc-a}
 \frac{\omega(a)}{\mathcal{M}\omega(j+1)}=\frac{1}{2}\int_{-\pi}^\pi\left(\frac{e^{\imath\varphi}}{a}\right)^jK_{\rm ev}^{(n)}([\omega];\sqrt{a},\sqrt{a}e^{-\imath\varphi})d\varphi,
\end{equation}
for any $a\in\mathbb{R}_+$ and  $j=0,\ldots,n-1$, and
\begin{equation}\label{rel-pol-calc-b}
 \frac{\mathcal{M}\omega(s)}{\mathcal{M}\omega(j+1)}=\int_0^\infty\left(\int_{-\pi}^\pi K_{\rm ev}^{(n)}([\omega];r,re^{-\imath\varphi})e^{\imath j\varphi}d\varphi\right)r^{2(s-j)}\frac{dr}{r},
\end{equation}
with $s\in{]s_{\min},s_{\max}[}+\imath\mathbb{R}$ and $j=0,\ldots,n-1$. The order of the integrations is important unless ${\rm Re}\,s>j$ and the Mellin transform also exists on the interval $]s_{\min},s_{\max}[$, namely then the integral is Lebesgue integrable. Both identities can be readily verified with the help of the explicit form of the kernel $K_{\rm ev}^{(n)}[\omega]$, see Eq.~\eqref{ker-ev}.

For the polynomials $p_l[\omega]$ we plug Eq.~\eqref{rel-pol-calc-b} into the first line of Eq.~\eqref{pol-sv} with $s=l+1$,
\begin{equation}\label{rel-pol-calc-c}
p_{l}([\omega];a)=\sum_{j=0}^l(-1)^{l-j}\frac{l!}{j!(l-j)!}a^j\int_0^\infty\left(\int_{-\pi}^\pi K_{\rm ev}^{(n)}([\omega];r,re^{-\imath\varphi})e^{\imath j\varphi}d\varphi\right)r^{2(l-j)+1}dr.
\end{equation}
 The finite sum in the index $j$ is a binomial sum and can be interchanged with the integrals yielding the desired result~\eqref{pol-rel}. In a similar way we plug Eq.~\eqref{rel-pol-calc-a} into the second line of Eq.~\eqref{func-sv} with $j=l$. Then we obtain the expression~\eqref{func-rel} for the functions $q_l[\omega]$.
 
We start from Eq.~\eqref{ker-sv-b} to derive the result for the relation between the kernels. To this end, we plug the first line of Eq.~\eqref{pol-sv} and the second line of Eq.~\eqref{func-sv} into Eq.~\eqref{ker-sv-b} and get
\begin{eqnarray}\label{rel-pol-calc-d}
K_{\rm sv}^{(n)}([\omega];a_b,a_c)=\int_0^1\sum_{j=0}^{n-1}\frac{(-1)^{n-j}(xa_b)^j}{j!(n-1-j)!}\frac{\partial_{a_c}^n\left[(-a_c)^n\omega(xa_c)\right]}{\mathcal{M}\omega(j+1)}dx.
\end{eqnarray}
The derivative can be pulled out of the integration because we integrate over a compact domain and the integrand is $n$-times continuous differentiable. Then we can apply the identity~\eqref{rel-pol-calc-a} for $a\to x a_c$ and we can perform the binomial sum. In the last step we rescale $x\to x/a_c$.

The second equality of Eq.~\eqref{ker-rel} is true because the integral
\begin{eqnarray}\label{rel-pol-calc-e}
&&\int_0^\infty\left(\int_{-\pi}^\pi K_{\rm ev}^{(n)}([\omega];\sqrt{x},\sqrt{x}e^{-\imath\varphi} )(a_c-a_be^{\imath\varphi})^{n-1}d\varphi\right)dx\\
&=&\int_0^\infty \omega(x)\sum_{j=0}^{n-1}\frac{2(n-1)!(-x a_b)^j a_c^{n-1-j}}{j!(n-1-j)!\mathcal{M}\omega(j+1)}dx\nonumber\\
&=&2(a_c-a_b)^{n-1}\nonumber
\end{eqnarray}
is finite and a polynomial of order $n-1$ in $a_c$. Thus the $n$th derivative in $a_c$ vanishes. This concludes the calculation.
\end{proof}

It is quite remarkable that all essential quantities for the singular value statistics are linearly dependent on the kernel of the eigenvalue statistics. Especially the integrals involved in the identities~\eqref{pol-rel} -- \eqref{ker-rel} can be interpreted as follows.  The integration over the angle $\varphi$ has to be expected. For example for normal matrices the singular values are nothing more than the projection of the eigenvalues onto the radial parts which is equal to an integration over the angles. A similar relation, although not completely the same due to the level repulsion for the singular values for bi-unitarily invariant matrices which might be missing for normal matrices, can be expected for the polynomial ensembles as well. The compact integration in the variable $x$ for the relation~\eqref{ker-rel} between the kernels can be also understood as a reminiscent property of Weyl's inequalities, cf. Eqs.~\eqref{inequalities-Weyl} and \eqref{inequalities-Horn}. On the one hand the integral over $x$ from $0$ to $a_c$ means that only the eigenvalues in the complex disk centered around the origin and of radius $a_c$ affect the singular values of value $a_c$. On the other hand the second identity in Eq.~\eqref{ker-rel} gives also the interpretation that only the eigenvalues outside this disk influence the statistics of a singular value at $a_c$. 
Despite this apparent conflict Eq.~\eqref{ker-rel} means that the eigenvalue statistics inside a certain disk and the ones outside this disk are intimately related. Note that this interpretation only applies for the ensembles considered in this section, namely polynomial ensembles of derivative type.

We are confident that the relations stated in Theorem~\ref{thm:rel-pol-SEV} will also carry over in one way or another to the relations between both kinds of statistics in the limit of large matrix dimensions. Certainly the Haagerup-Larson theorem~\cite{HL-theorem,HS:2007} restricted to positive polynomial ensembles of derivative type has to follow as well as the single ring theorem~\cite{SR-theorem-1,SR-theorem-2}. However we will not study this limit since it will exceed the aim of the present work.
 
% ****************************************************************************************************
 
\section{Conclusions and Outlook}
\label{sec:Conclusions}
 
We have discovered a remarkable relation between the joint densities of the eigenvalues and of the singular values of a bi-unitarily invariant matrix ensemble. Due to this relation we have not only shown that the non-compact integration over the unitriangular matrices in the Schur decomposition \eqref{schur-dec} is invertible for this kind of ensembles, but we have also derived an explicit map between the two kinds of spectral densities, see Theorem~\ref{thm:Map}. Thus, we have opened a novel approach which allows to directly relate statistical quantities of the singular values with those of the eigenvalues. We have illustrated this via a certain class of polynomial ensembles which we call {\it polynomial ensembles of derivative type}, see Definition~\ref{def:PE}.b) and Theorem~\ref{thm:rel-pol-SEV} for the explicit relations between the two spectral statistics. A certain subset of this class called Meijer G-ensembles was already encountered in the discussion of products of certain random matrices~\cite{ARRS:2013,AB:2012,AIK:2013,AKW:2013,ABKN,Forrester:2014,IK:2013,KS:2014,KZ:2014,KKS:2015}. Additionally the Muttalib-Borodin ensembles~\cite{Muttalib,Borodin,Forrester-Wang} of the Laguerre- and the Jacobi-type are also polynomial ensembles of derivative type. Therefore our result answers the question what is the corresponding joint density of eigenvalues of the matrix $ak$ when $a$ is a Muttalib-Borodin ensemble (or any other polynomial ensemble of derivative type) and $k$ an independent Haar-distributed unitary matrix.

Thus we are now able to address the long-standing problem of describing the relation between the distributions
of the eigenvalues and of the singular values. Previous solutions to this problem in form of the Haagerup-Larsen theorem~\cite{HL-theorem,HS:2007} and the single ring theorem~\cite{SR-theorem-1,SR-theorem-2} refer to the limit of large matrices via free probability. Our approach is exact at finite matrix dimensions. Hence the local spectral statistics of the singular values and the eigenvalues can be studied as well, via these exact relations we discovered. Our approach should even open the opportunity to study mixed statistics. 
For example one can ask for the joint density of $k_{\rm ev}$ eigenvalues and $k_{\rm sv}$ singular values.

We have also shown that for certain deformations of bi-unitarily invariant ensem\-bles, called $G$-adjoint-invariant deformations, see Definition~\ref{def:deform}, one can generalize the relation between the joint densities of the eigenvalues and of the singular values. This relation is not invertible in contrast to that for bi-unitarily invariant random matrix ensembles. Nonetheless we can explicitly answer the question what the joint density of the singular values is when the joint density of the eigenvalues has a certain form, see Theorem~\ref{thm:deform}. Hence one is now able to study random matrix ensembles defined by the joint density of their eigenvalues not only as normal matrix ensembles~\cite{CZ:1998,TBAZW:2005,BK:2012} but as deformations of bi-unitarily invariant matrix ensembles. Those ensembles have non-trivial singular value statistics due to the level repulsion of the singular values which is quite often absent for normal random matrix models.

Our approach is based on harmonic analysis on matrix spaces, especially on the spherical transform. 
As is well known, one of the key properties of this transform 
is its factorization property with respect to multiplicative convolutions
\cite{HarCha1,HarCha2,Helgason3,Terras,JL:SL2R}.
Indeed, in the recent work \cite{KK-2016}, we use this connection to investigate products of independent random matrices from polynomial ensembles.

One mathematical question is still open and has to be answered. 
Our approach for bi-unitarily invariant ensembles shows that 
a joint probability density and, hence, positive density on the ``singular values'' 
automatically carries over to a joint probability density on the ``eigenvalues'',
whereas the reverse direction may fail for matrix dimension $n > 1$.
Therefore,
$\mathcal{R}(L^{1,{\rm SV}}_{\text{prob}}(A))$ is in general a proper subset of $L^{1,{\rm EV}}_{\text{prob}}(Z)$, 
where $L^{1,{\rm SV}}_{\text{prob}}(A)$ and $L^{1,{\rm EV}}_{\text{prob}}(Z)$ are the sets of the joint probability densities of singular values and eigenvalues. This also implies that not every joint probability density of ``eigenvalues'' corresponds to a bi-unitarily invariant random matrix ensemble. Note, however, that this does not mean that there exist no other random matrix ensembles yielding these joint probability densities of eigenvalues.

{In particular, it is quite likely that a joint probability density for the ``eigenvalues'' 
may correspond to a signed density for the ``singular values''.
This observation is a bit surprising and unfortunate for a perfect correspondence for probability densities 
between the two kinds of spectral statistics. The question is: What is the image $\mathcal{R}(L^{1,{\rm SV}}_{\text{prob}}(A))$ of our map restricted to joint probability densities for the singular values of bi-unitarily random matrices? Or, in other words, what conditions do joint probability densities of the eigenvalues have to satisfy so that they correspond to a joint probability density for the singular values? This question seems to be quite tough and should be addressed in a forthcoming publication.

Another question which is quite intriguing is the generalization of our results to real and quaternion matrices. One might at least speculate that the (more complicated) analogue
of Eq.~\eqref{eq:gelfand-naimark} for the orthogonal and symplectic groups
might prove useful in studying the relation among the eigenvalue and singular value statistics 
beyond the complex case, which is still open.

%****************************************************************************************************
\section*{Acknowledgements}

We want to thank Gernot Akemann, Friedrich G\"otze and Arno Kuijlaars for fruitful discussions on this topic. Moreover we acknowledge financial support by the {\it CRC 701: ``Spectral Structures and Topological Methods in Mathematics''} of the {\it  Deutsche Forschungsgemeinschaft}.

\end{document}